\documentclass{amsart}

\usepackage[normalem]{ulem}
\usepackage{amsmath,stackrel,enumerate}
\usepackage{xcolor} %options= blue, brown, green lime, magenta, olive, orange, violet, red, teal, orange, pink
\usepackage{xypic}
\usepackage{amssymb}

\usepackage[backref=page]{hyperref}
\newcommand{\bu}{\mathbf{u}}
\newcommand{\bv}{\mathbf{v}}
\newcommand{\bw}{\mathbf{w}}

\def\N{{\mathbb{N}}}
\def\Q{{\mathbb{Q}}}
\def\Z{{\mathbb{Z}}}

\DeclareMathOperator{\Apr}{Apr}
\DeclareMathOperator{\tor}{tor}
\DeclareMathOperator{\gp}{Grp}
\DeclareMathOperator{\lcm}{lcm}
\DeclareMathOperator{\rk}{rank}

\def\uRing#1{\kappa[\bu^{#1}]}

\def\fibersum#1#2#3{#1 \oplus_{#3} #2}  %fibered sum
\def\Fibersum#1#2#3{#1 \widetilde{\oplus}_{#3} #2}  %fibered sum for affine semigroups
\def\tensor#1#2#3{#1 \otimes_{#3} #2}  %tensor product
\def\Tensor#1#2#3{#1 \widetilde{\otimes}_{#3} #2}

\newtheorem{thm}{Theorem}[section]
\newtheorem{cor}[thm]{Corollary}
\newtheorem{lem}[thm]{Lemma}
\newtheorem{prop}[thm]{Proposition}

\newtheorem{notation}[thm]{Notation}

\theoremstyle{definition}
\newtheorem{defn}[thm]{Definition}
\newtheorem{example}[thm]{Example}

\makeatletter
\@namedef{subjclassname@2020}{\textup{2020} Mathematics Subject Classification}
\makeatother
\begin{document}

%%%%%%%%%%%%%%%%%%%%%%%%%%%%%%%%%%%%%%%%%%%%%%%%%%%%%

\title[Affine Semigroup Algebras]{Affine Semigroup Algebras And Their Fibered Sums}
\author{C-Y. Jean Chan \and I-Chiau Huang \and Jung-Chen Liu}

\address{Department of Mathematics, Central Michigan University, Mt. Pleasant, MI 48859, U.S.A.}
\email{chan1cj@cmich.edu}

\address{Institute of Mathematics, Academia Sinica, Taipei, Taiwan.}
\email{ichuang@math.sinica.edu.tw}

\address{Department of Mathematics, National Taiwan Normal University, Taipei, Taiwan.}
\email{liujc@math.ntnu.edu.tw}

\begin{abstract}
We study affine semigroup rings as algebras over subsemigroup rings.
From this relative viewpoint with respect to a given subsemigroup ring, the fibered sum of two affine semigroup algebras is constructed. 
Such a construction is compared to the tensor product and to the classical gluings of affine semigroup rings as defined in Rosales\cite{ros:psN}.

While fibered sum can always be achieved, gluings of affine semigroup rings do not always exist. 
Therefore, we further investigate when the fibered sum of affine semigroup algebras gives rise to a gluing. 
A criterion is recovered in terms of the defining semigroups under which the gluing may take place. 
\end{abstract}


\maketitle

%%%%%%%%%%%%%%%%%%%%%%%%%%%%%%%%%%%%%%%%%%%%%%

\section{Introduction}\label{sec:intro}

%%%%%%%%%%%%%%%%%%%%%%%%%%%%%%%%%%%%%%%%%%%

The main aim of this paper is to construct  fibered sums in the
category of affine semigroup rings  as our perception regarding gluing of affine semigroup rings. 
Our work is inspired by a question raised by Gimenez and Srinivasan~\cite{gim-srin:gswh} about when and how affine semigroup rings can be glued.
We recount the history of gluing in the literature below, following a brief introduction of affine semigroup rings.

In geometric perspective, an affine semigroup ring with the semigroup embedded in $\Z^d$ represents 
regular functions on an affine variety parametrized by $d$ variables.
For instance, numerical semigroup rings are associated with monomial curves. 
Higher dimensional affine semigroup rings can be in some cases associated with affine toric varieties.
Due to the multiple layers of interpretations, affine semigroup rings have been of great interests to many disciplines in
mathematics.

An affine semigroup ring is usually regarded as a finitely generated algebra over a field $\kappa$.
Besides such an absolute viewpoint, an affine semigroup ring can also be regarded as an algebra over another ring  
obtained from a subsemigroup.
The latter is what we call a relative viewpoint.
Works along this line among other topics have been done by the second author, Kim and
Jafari~\cite{hu-jaf:fnsg, hu-jaf:crnsa, hu-kim:nsa} in one-dimensional cases, that is,
for numerical semigroup rings.

In this paper, we implement our relative viewpoint  on affine semigroup rings to the concept of gluing, which was introduced originally in absolute cases.
Gluing of numerical semigroup rings in a special case appeared in Watanabe~\cite[Lemma~1]{wat:seodgd} following
the study of complete intersection numerical semigroup rings by Herzog~\cite{her:grassr}.
Delorme~\cite{del:smicn} treated gluing systematically for numerical semigroups and
obtained a thorough characterization for complete intersections.
Gluing was generalized to affine semigroup rings by Rosales~\cite{ros:psN}.
It was further studied in Gimenez and Srinivasan~\cite{gim-srin:gswh,gim-srin:gsNc}.

In the approach just mentioned above, gluing is defined for arbitrary two numerical semigroup rings.
However, gluing is not a general operation in higher dimensions; it exists only between those pairs satisfying certain conditions.
More precisely, let $S_1$ and $S_2$ be two semigroups in $\N^d$ for some positive $d$, and
consider their associated affine semigroup rings $\kappa[S_1]$ and $\kappa[S_2]$ over a field $\kappa$.
The gluing of $\kappa[S_1]$ and $\kappa[S_2]$ is defined if there exist two positive integers $a$ and $b$
such that
some specific identification holds between  $\kappa[aS_1]$ and $\kappa[bS_2]$
(see (\ref{glue_varphi}) and (\ref{glueKernel}) in Section~\ref{sec:gluing}), where
$aS_1$ is the rescaling of $S_1$ by the constant $a$, and similarly for $bS_2$.
In such a case, the gluing produces a larger affine semigroup ring $\kappa[a S_1 + b S_2]$,
which is not necessarily isomorphic to $\kappa[S_1+ S_2]$, even though $\kappa[aS_1]$ is  isomorphic to $\kappa[S_1]$ and so is $\kappa[bS_2]$ to $\kappa[S_2]$.
For example, the affine semigroup ring
$A=\kappa[x^3,x^2y,xy^2,y^3]$ can not be glued to itself as shown in
\cite[Example 1]{gim-srin:gswh}.
On the other hand,
$A$ has another two isomorphic presentations
$A_1=\kappa[x^3y^3,x^3y^2z,x^3yz^2,x^3z^3]$ and $A_2=\kappa[x^4,x^3y,x^2y^2,xy^3]$
(see Example~\ref{ex:2copies}).
The latter two copies $A_1$ and $A_2$ can be glued \cite[Example 7]{gim-srin:smfrsrog}.
It is crucial to observe that $A$ is a subring of a polynomial ring in two variables while a polynomial ring of three variables is required to house both $A_1$ and $A_2$ simultaneously.
This example illustrates that the existence of gluing depends on the relative position of the two rings, or precisely, how the affine semigroup rings are embedded.

Through another set of rings in Example~\ref{ex:normalcurve2} and the statements following it,
we will provide detailed explanations about the contrast of the phenomenon presented by $A$ versus $A_1$ and $A_2$ just described.
In particular, while an adjustment of rescaling by $a$ and $b$ is allowed for the purpose of possible gluing,
we see that if matrix multiplications are utilized instead of scalar multiplications,
then one can transform the affine semigroup rings into an appropriate ambient polynomial ring
where the gluing may take place. The resulting glued affine semigroup rings satisfy certain universal property. This leads to the fibered sum in our consideration.

In other words, the fibered sum of affine semigroup algebras constructed in this paper provides a more general perspective than gluing.
The consideration of algebras incorporates a coefficient ring which can be any affine subsemigroup ring in general.
Gluing is a special case of fibered sum where the coefficient ring is isomorphic to a polynomial ring in one variable.
We investigate also when the fibered sum of affine semigroup rings gives rise to a gluing.
An equivalent condition to the availability of gluing two affine semigroup rings with only scalar multiplications is worked out in Theorem~\ref{glueEqCond} recovering a criterion in \cite[Theorem 1.4]{ros:psN}.

We see in Theorem~\ref{thm:Fibersumetensor} that, under certain conditions, 
the fibered sum is the same as a tensor product. 
Discussions in this regard can be found in Examples~\ref{ex:notFlat} and \ref{ex:folding}.

Since affine semigroup rings can be viewed as
the exponential counterpart of affine semigroups,
our construction is built from the fibered sum of affine semigroups.
To make these statements precise, we begin with some definitions which will be used throughout the paper.

A {\em monoid} is a nonempty set with an associative binary operation and an identity element.
Since the inverse of any given element may not exist, a monoid is not necessarily a group.
The binary operator of a monoid under our consideration is always commutative.
An {\em affine semigroup} is a finitely generated monoid $S$ that satisfies the following intrinsic properties:
for any $s, s_1, s_2 \in S$,
\begin{enumerate}
\item[(i)] ({\em cancellative}) if $s + s_1 = s + s_2$, then $s_1 = s_2$;
\item[(ii)] ({\em torsion free}) if $n s_1 = n s_2$  in $S$ for some positive integer $n$, then $s_1 = s_2$.
\end{enumerate}
As seen in \cite[Corollary 3.4]{ros-gar:fgcm},
a finitely generated monoid is an affine semigroup if and only if it can be embedded in $\Z^d$
(equivalently in $\Q^d$) for some positive integer $d$.
We will see in Section~\ref{sec:monoid} that the cancellative and torsion free conditions are essential in many discussions.
As part of the definition, our affine semigroup always contains an identity element. 
Following the modern convention, we use $\N$ to denote the set of non-negative integers.

Let $\kappa$ be a field. Associated to an affine semigroup $S$, an {\em affine semigroup ring} is a 
$\kappa$-algebra generated by monomials with exponents in $S$.
More precisely, by embedding $S$ in $\Z^d$, we may identify each element $s \in S$ with a $d$-tuple $(s_1, \dots, s_d) \in \Z^d$
and consider the $\kappa$-algebra
\[\uRing S = \kappa [ \{ \bu^s=u_1^{s_1} \cdots u_d^{s_d} \mid s=(s_1, \dots, s_d) \in S\}] . \]
Obviously, the $\kappa$-algebra
$\uRing S$ is finitely generated and since it is a subring of the Laurent polynomial ring
$\kappa[ u_1^{\pm 1}, \dots, u_d^{\pm 1}]$, it is also an integral domain.
Henceforth, we will use $\kappa[\bu^S]$ to denote an affine semigroup ring instead of $\kappa[S]$.
It is essential to distinguish various embeddings of an affine semigroup ring.
The use of variables $\bu$ helps keep track of  changes of embeddings.
Such necessity will become clear in Sections~\ref{sec:algebra} and \ref{sec:gluing}, especially in the latter one. 

The paper is organized as follows:
Section~\ref{sec:monoid} first establishes fibered sums in the category of cancellative monoids.
A fibered sum of affine semigroups turns out to be a quotient of that of cancellative monoids.
Corollary~\ref{cor:emb} offers a glimpse of fibered sums in a special case: assuming that
$S_1$, $S_2$ and $S$ are subsemigroups contained in an affine semigroup $T$ with $S \subset S_1\cap S_2$,
then $S_1 + S_2$ in $T$ is the fibered sum of $S_1$ and $S_2$ over $S$ provided that $\gp(S_1) \cap \gp(S_2) = \gp(S)$, where
$\gp( \cdot )$ denotes the group generated by the argument in it.
The key fact that proves this statement is also used in proving Lemma~\ref{lem:gluingCond} and Theorem~\ref{glueEqCond} that provide a condition enabling the gluing to take place.

 An affine semigroup $S$ is {\em positive} if the only subgroup contained in $S$ is the trivial group.
In particular, a numerical semigroup is positive.
In Section~\ref{sec:Apery}, we extend the notion of Ap\'{e}ry elements from numerical semigroups to positive
affine semigroups.
The notion is relative in nature, with which we may talk about
representations of an element with respect to an affine
subsemigroup. This also leads to a criterion for flatness discussed in Section~\ref{sec:algebra}.

Built from the fibered sums of affine semigroups,
Section~\ref{sec:algebra} is dedicated to affine semigroup algebras and their fibered sums.
If $R'$ is a positive affine semigroup algebra over $R$, then, using Ap\'ery elements,
Theorem~\ref{prop1:freeflat} shows that $R'$ is free over $R$ if and only if it is flat over $R$.
We also explore the comparison between fibered sums and the usual tensor products.
 It is important to be reminded that the tensor product of two affine semigroup algebras is not necessarily an affine semigroup algebra.
Theorem~\ref{thm:Fibersumetensor} proves that under certain conditions, the fibered sum of affine semigroup algebras is isomorphic to the usual tensor product.

Section~\ref{sec:gluing} offers an overall comparison between fibered sums and the gluing defined by Rosales~\cite{ros:psN}.
As mentioned previously, a fibered sum is a more general notion and can be done between arbitrary
two affine semigroup algebras with a prescribed common coefficient ring.
Our study of fibered sums is inspired by the works of Gimenez and Srinivasan~\cite{gim-srin:gswh,gim-srin:gsNc}.
We discover that gluing, when available, satisfies certain universal property.
In fact, gluing is a special case of fibered sums when the coefficient ring is taken as a polynomial ring in one variable 
(see Theorem~\ref{thm:glueVfibersum}).
We also prove in Lemma~\ref{lem:gluingCond} that whether or not a gluing
exists depends on the relationship among the groups generated by the semigroups under consideration.

%%%%%%%%%%%%%%%%%%%%%%%%%%%%%%%%%%%%%%%%%%%%%

\section{Fibered Sums}\label{sec:monoid}

%%%%%%%%%%%%%%%%%%%%%%%%%%%%%%%%%%%%%%%%%%%%%

Through abelian groups, we construct and investigate fibered sums of cancellative monoids and those of affine semigroups in this section. In a category, the {\em fibered sum}
of morphisms $S\to S_1$ and $S\to S_2$ consists
of an object $S_1\sqcup_S S_2$ together with morphisms $S_1\to S_1\sqcup_S S_2$ and
$S_2\to S_1\sqcup_S S_2$ such that the diagram
\[
\xymatrix{
S_1  \ar[r]&S_1\sqcup_S S_2\\
S \ar[u]\ar[r] & S_2\ar[u]   }
\]
is commutative and universal in the following sense: for any commutative diagram
\[
\xymatrix{
S_1  \ar[r]&T\\
S \ar[u]\ar[r] & S_2\ar[u]   }
\]
in the category, there exists a unique morphism $S_1\sqcup_S S_2\to T$ making the
diagram
\[
\xymatrix{
&& T\\
S_1  \ar[r]\ar@/^1pc/[rru]&S_1\sqcup_S S_2\ar[ru]\\
S \ar[u]\ar[r] & S_2\ar[u]\ar@/_1pc/[ruu]   }
\]
commute. We also say that $S_1\sqcup_S S_2$ is the
{\em fibered sum of $S_1$ and $S_2$ over $S$.}

In this paper, a monoid $S$ with the associative binary operation $+$ and an identity
element $0$ is always assumed to be commutative.
A monoid is {\em cancellative} if
$s_1=s_2$ whenever $s_1+s_3=s_2+s_3$ for $s_1,s_2,s_3 \in S$.
Every cancellative monoid $S$ is naturally contained in its {\em group of differences}
$\gp(S)=\{s_1-s_2 \mid s_1,s_2\in S\}$ with the binary operation inherited from $S$.
A homomorphism $S\to S'$ of cancellative monoids induces a
homomorphism of abelian groups $\gp(S)\to\gp(S')$. We often denote the image of an element
$s\in S$ in $S'$ also by $s$. Within a commutative diagram of monoids, such abuse
of notation does not cause confusion.
A monoid $S$ is {\em torsion free} if for every $s_1,s_2\in S$ and positive integer $n$,
the condition $ns_1=ns_2$ implies $s_1=s_2$.
Given a monoid $S$, we define an equivalence relation: for $s_1,s_2\in S$,
\[
s_1\sim_{\tor} s_2 \iff  \text{ $\exists$   $n>0$ in $\Z$ such that } ns_1=ns_2.
\]
The quotient $S/ \sim_{\tor}$ is a torsion free monoid. If a monoid $S$ is
cancellative, then so is $S/\sim_{\tor}$.

\begin{example}
Any submonoid of the monoid $\N$ of nonnegative integers is finitely generated, cancellative and torsion free. 
The submonoid $\{(x,y) \mid x,y>0\}\cup\{(0,0)\}$ of
$\N^2$ is cancellative and torsion but not finitely generated.
\end{example}

An {\em affine semigroup} is a finitely generated monoid isomorphic to a submonoid of $\Q^d$. 
An affine semigroup that can be embedded into $\N$ is called a {\em numerical semigroup}.  
In the literature, the greatest common divisor of a numerical semigroup is usually assumed to be $1$. 
In our definition, we lift the restriction on the greatest common divisor to allow embeddings by rescaling.
Multiplied by an integer, an affine semigroup $\Q^d$ can be embedded into some $\Z^d$.
It is known that a finitely generated commutative monoid is an affine semigroup if and
only if it is cancellative and torsion free \cite[Corollary 3.4]{ros-gar:fgcm}.
For an affine semigroup $S$, the group $\gp(S)$ is a free abelian group,
whose rank is denoted by $\rk(S)$.

The first main result of this section is the existence of fibered sums in the category of cancellative monoids, and then by modulo
the equivalence relation $\sim_{\tor}$ defined above,
we obtain fibered sums of affine semigroups.
Let $S \to S_1$ and $S \to S_2$ be homomorphisms of cancellative monoids.
Regardless of $S$, the direct sum of $S_1$ and $S_2$ is denoted by $\fibersum{S_1}{S_2}{}$ as usual.
We shall construct a cancellative monoid $\fibersum{S_1}{S_2}{S}$, which is a quotient of  $\fibersum{S_1}{S_2}{}$, and show that it is a fibered sum in the category of cancellative monoids.
If the monoids $S, S_1, S_2$ are affine semigroups, we shall construct an affine semigroup $\Fibersum{S_1}{S_2}{S}$ and show that it is
a fibered sum in the category of affine semigroups.

Now, we begin to construct fibered sums in the category of cancellative monoids.
The homomorphisms $S\to S_1$ and $S\to S_2$ of cancellative monoids
 induce
the fibered sum $\gp(S_1)\sqcup_{\gp(S)}\gp(S_2)$ in the category of abelian groups
that is isomorphic to the direct sum $\gp(S_1)\oplus \gp(S_2)$ modulo the
subgroup consisting of elements $(s,-s)$, where $s\in\gp(S)$.

\begin{notation}\label{notation:oplus}
Let $S$, $S_1$ and $S_2$ be cancellative monoids.
The notation
$\fibersum{S_1}{S_2}{S}$ denotes the image of the
composition of the homomorphisms
\begin{equation}\label{eq:96881}
\fibersum{S_1}{S_2}{}\to \gp(S_1)\oplus \gp(S_2) \to
\gp(S_1)\sqcup_{\gp(S)}\gp(S_2).
\end{equation}
Given an element $(s_1,s_2)$ in the direct sum
$\fibersum{S_1}{S_2}{}$, its image in $\fibersum{S_1}{S_2}{S}$ is denoted by
$\fibersum{s_1}{s_2}{}$.
\end{notation}

 If $(s_1, s_2)$ and $(t_1,t_2)$ in $S_1 \oplus S_2$ have the same image, namely
$\fibersum{s_1}{s_2}{}=\fibersum{t_1}{t_2}{}$ in $\gp(S_1)\sqcup_{\gp(S)}\gp(S_2)$,
then we must have $(s_1, s_2) = (t_1, t_2) + (s, -s)$ in
$\gp(S_1)\oplus \gp(S_2)$ for some $s \in \gp(S)$. In this case, we may write $s = b-a$ for some $a,b \in S$, and have
\[
(s_1, s_2) + (a, -a) = (t_1, t_2) + ( b, -b)
\]
in $\gp(S_1)\oplus \gp(S_2)$, which is equivalent to
\[
(s_1, s_2) + (a, 0) + (0,b) = (t_1, t_2) + ( b, 0) + (0, a).
\]
Now the last equality also holds in $S_1 \oplus S_2$.
This gives the following equivalent conditions in terms of elements in the cancellative monoids:
\begin{equation}\label{eq:47125}
\fibersum{s_1}{s_2}{}=\fibersum{t_1}{t_2}{} \iff
\begin{cases}
s_1+a=t_1+b \quad \mbox{in } S_1 \\s_2+b=t_2+a \quad \mbox{in } S_2
\end{cases}
\mbox{for some $a,b\in S$.}
\end{equation}
In particular, $\fibersum{a}{0}{}=\fibersum{0}{a}{}$ in
$\fibersum{S_1}{S_2}{S}$ for any $a \in S$. In other words, the diagram
\[
\xymatrix{
S_1  \ar[r]&\fibersum{S_1}{S_2}{S}\\
S \ar[u]\ar[r] & S_2\ar[u]   }
\]
is commutative, where $S_i\to \fibersum{S_1}{S_2}{S}$ is the composition of the
canonical embedding $S_i\to \fibersum{S_1}{S_2}{}$ and the homomorphisms in (\ref{eq:96881}).

\begin{lem}\label{lem:fiberedSum}
In the category of cancellative monoids, given homomorphisms $S\to S_1$ and $S\to S_2$,
the monoid $\fibersum{S_1}{S_2}{S}$ together with the canonical homomorphisms $S_1\to \fibersum{S_1}{S_2}{S}$ and $S_2\to \fibersum{S_1}{S_2}{S}$ is the fibered sum of
$S_1$ and $S_2$ over $S$.
\end{lem}

\begin{proof}
We note that the cancellative property holds in $\fibersum{S_1}{S_2}{S}$  since it is a homomorphic image in a group.
Thus it belongs to the category of cancellative monoids.

To check that $\fibersum{S_1}{S_2}{S}$ is the fibered sum, we consider a commutative
diagram
\[
\xymatrix{
&& T\\
S_1  \ar[r]\ar@/^1pc/[rru]^{\varphi_1}&\fibersum{S_1}{S_2}{S} \\
S \ar[u]\ar[r] & S_2\ar[u]\ar@/_1pc/[ruu]_{\varphi_2}   }
\]
of cancellative monoids. Since $\fibersum{S_1}{S_2}{S}$ is generated by the images of
$S_1$ and $S_2$, there is at most one homomorphism $\fibersum{S_1}{S_2}{S}\to T$
that extends the above commutative diagram. 
To verify such a homomorphism is well-defined,
we need to show that
if $\fibersum{s_1}{s_2}{}=\fibersum{t_1}{t_2}{}$ in $\fibersum{S_1}{S_2}{S}$ then
\begin{equation}\label{eq:6884}
\varphi_1(s_1)+\varphi_2(s_2)=\varphi_1(t_1)+\varphi_2(t_2).
\end{equation}
There exist $a,b\in S$ such that $s_1+a=t_1+b$ and $s_2+b=t_2+a$. 
Applying $\varphi_1$ and $\varphi_2$ and summing up their images in $T$, we have
\begin{equation}\label{eq:4887}
\varphi_1(s_1)+\varphi_1(a)+\varphi_2(s_2)+\varphi_2(b)=
\varphi_1(t_1)+\varphi_1(b)+\varphi_2(t_2)+\varphi_2(a).
\end{equation}
From the commutative diagram, $\varphi_1(a)=\varphi_2(a)$ and
$\varphi_1(b)=\varphi_2(b)$. Since $T$ is cancellative, we may cancel out
images of $a$ and $b$ from (\ref{eq:4887}) to obtain the required identity
(\ref{eq:6884}).
\end{proof}

Next, we construct the fibered sum in the category of affine semigroups.  Let  $S\to S_1$ and $S\to S_2$ be homomorphisms of affine
semigroups. Since $\fibersum{S_1}{S_2}{S}$ is a finitely generated cancellative monoid,
its quotient $(\fibersum{S_1}{S_2}{S})/\sim_{\tor}$ is also a finitely generated cancellative monoid and thus
an affine semigroup since it is obviously torsion free.

\begin{notation}\label{notation:oplus+}
Let $S$, $S_1$ and $S_2$ be affine semigroups described above.
We use the notation
\[
\Fibersum{S_1}{S_2}{S} := (\fibersum{S_1}{S_2}{S})/\sim_{\tor}
\]
for the quotient of $\fibersum{S_1}{S_2}{S}$
modulo the equivalence relation.
Moreover, we write $\Fibersum{s_1}{s_2}{}$ for the image of $\fibersum{s_1}{s_2}{} \in \fibersum{S_1}{S_2}{S}$ in $\Fibersum{S_1}{S_2}{S}$.

\end{notation}

\begin{prop}\label{prop:UP}
Let $S$, $S_1$ and $S_2$ be as in Notation~\ref{notation:oplus+}.
The affine semigroup $\Fibersum{S_1}{S_2}{S}$
is the fibered sum of $S_1$ and $S_2$ over $S$ in the category of
affine semigroups.
\end{prop}
\begin{proof}
Consider a commutative diagram
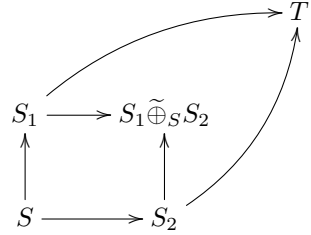
\[
\xymatrix{
&& T\\
S_1  \ar[r]\ar@/^1pc/[rru]&\Fibersum{S_1}{S_2}{S} \\
S \ar[u]\ar[r] & S_2\ar[u]\ar@/_1pc/[ruu]  }
\]
of affine semigroups. As cancellative monoids, there is a unique homomorphism
$\varphi\colon \fibersum{S_1}{S_2}{S}\to T$ making the diagram
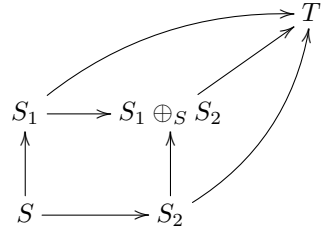
\[
\xymatrix{
&& T\\
S_1  \ar[r]\ar@/^1pc/[rru]&\fibersum{S_1}{S_2}{S} \ar[ru]\\
S \ar[u]\ar[r] & S_2\ar[u]\ar@/_1pc/[ruu]  }
\]
commute. Since $T$ is torsion free, the homomorphism $\varphi$ induces a homomorphism
$\Fibersum{S_1}{S_2}{S}\to T$, which is the unique one making the diagram
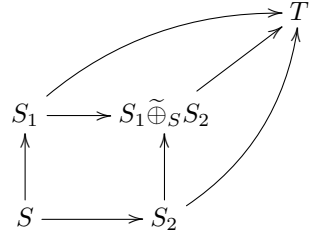
\[
\xymatrix{
&& T\\
S_1  \ar[r]\ar@/^1pc/[rru]&\Fibersum{S_1}{S_2}{S} \ar[ru]\\
S \ar[u]\ar[r] & S_2\ar[u]\ar@/_1pc/[ruu]  }
\]
commute.
\end{proof}

With Proposition~\ref{lem:fiberedSum} and Proposition~\ref{prop:UP},
we now formally give Notation~\ref{notation:oplus} and Notation~\ref{notation:oplus+} their anticipated titles.

\begin{defn}\label{def:fibersum}
In the category of cancellative monoids, let $S$, $S_1$ and $S_2$ be objects with homomorphisms
$S \rightarrow S_1$ and $S \rightarrow S_2$.
The cancellative monoid $\fibersum{S_1}{S_2}{S}$ is called the {\em fibered sum of $S_1$ and $S_2$ over $S$}. 
Futhermore if $S$, $S_1$ and $S_2$ are affine semigroups, then
$\Fibersum{S_1}{S_2}{S}$ is called the {\em fibered sum of $S_1$ and $S_2$ over $S$} in the category of
affine semigroups.
\end{defn}

Properties of fibered sums depend on the underlying monoid $S$.

\begin{example}\label{ex:6Nvs12N}
It can be checked directly by the universal property that
$\fibersum{2\N}{3\N}{6\N} \simeq \langle 2,3\rangle\subset \N$. Moreover,
$\Fibersum{2\N}{3\N}{6\N} \simeq  \fibersum{2\N}{3\N}{6\N}$
since $\langle 2,3 \rangle$ is already torsion free.
On the other hand in $\fibersum{2\N}{3\N}{12\N}$, we have 
$\fibersum{6}{0}{} \neq \fibersum{0}{6}{}$ while $2 (\fibersum{6}{0}{}) = 2 (\fibersum{0}{6}{})$, so
$\fibersum{2\N}{3\N}{12\N}$ is not torsion free.  Hence, $\Fibersum{2\N}{3\N}{12\N} \not\simeq \fibersum{2\N}{3\N}{12\N}$.
\end{example}

%{\cyan 
%Next we make a cautious remark {\brown (or reminder)} that the intersection of two arbitrary monoids $S_1$ and $S_2$
%is not well-defined unless they are are simultaneously embedded in a larger monoid. 
%} 
If $S\to S_1$ and $S\to S_2$ are embeddings of affine semigroups, 
so are the canonical homomorphisms $S_1\to\Fibersum{S_1}{S_2}{S}$ and
$S_2\to\Fibersum{S_1}{S_2}{S}$. 
It is always true that  $\Fibersum{S_1}{S_2}{S}$
is the sum of the images of $S_1$ and $S_2$ under the canonical homomorphisms.
In such a case, $S_1$ and $S_2$ can be considered as submonoids of $\Fibersum{S_1}{S_2}{S}$.  
Furthermore,
embeddings of $S$ into $\Fibersum{S_1}{S_2}{S}$
via $S_i$ for $i=1, 2$ agree as well, so
$S$ can  be considered as a submonoid of $\Fibersum{S_1}{S_2}{S}$.
Through these embeddings, the intersection of $S_1$ and $S_2$ may take place
inside $\Fibersum{S_1}{S_2}{S}$. We may consider also the fibered sum of $S_1$ and $S_2$ over other affine semigroups such as $S_1 \cap S_2$ or any of its submonoids in $\Fibersum{S_1}{S_2}{S}$.
In fact by inspecting the universal property in the category of affine semigroups,
$\Fibersum{S_1}{S_2}{S}$ is also the fibered sum of $S_1$ and $S_2$ over any
submonoid of $S_1 \cap S_2$ containing $S$.
For instance,  $\Fibersum{2\N}{3\N}{12\N} \simeq  \Fibersum{2\N}{3\N}{6\N}$ in Example~\ref{ex:6Nvs12N}.
See Example~\ref{ex:algebra6Nvs12N} for the $\kappa$-algebra version of this statement.

Note that even if $S_1$, $S_2$ and $S$ are submonoids of an affine semigroup $T$, the fibered sum $\Fibersum{S_1}{S_2}{S}$ may not be embeddable into $T$. It may happen that $S_1+S_2$ taken in $T$ is not isomorphic to $\Fibersum{S_1}{S_2}{S}$. See Example~\ref{ex:51661}. The next proposition gives an intrinsic condition with which the fibered sum can be obtained inside a larger affine semigroup.

\begin{prop}\label{prop:58908}
Let $S\to S_1$ and $S\to S_2$ be embeddings of affine semigroups. Assume that
$S_1$ and $S_2$ have the same rank as $S$. If $S_1$ and $S_2$ embed
to an affine semigroup $T$, whose compositions following the above embeddings agree on $S$,  then $S_1+S_2$ taken in $T$ is isomorphic to $\Fibersum{S_1}{S_2}{S}$
\end{prop}
\begin{proof}
The point is that the condition on the ranks is equivalent to
$\gp(nS_1)\subset\gp{S}$ and $\gp{(nS_2)}\subset\gp{S}$ for some $n>0$.
We regard $S_1$, $S_2$ and $S$ as subsets of $T$. The universal property
of fibered sum gives a map $\Fibersum{S_1}{S_2}{S} \to S_1+S_2$ with $\Fibersum{s_1}{s_2}{} \mapsto s_1+s_2$, which
is clearly onto. To show that the map is one-to-one, we consider elements
$\Fibersum{s_1}{s_2}{}$ and $\Fibersum{t_1}{t_2}{}$ in $\Fibersum{S_1}{S_2}{S}$ such that $s_1+s_2=t_1+t_2$.
With the same $n$ above, $ns_1-nt_1=nt_2-ns_2 \in \gp{S}$. Then $\fibersum{ns_1}{ns_2}{} = \fibersum{nt_1}{nt_2}{}$ and this implies $\Fibersum{s_1}{s_2}{} = \Fibersum{t_1}{t_2}{}$.
\end{proof}

\begin{cor}\label{numericalFibersum}
Let $S_1$ and $S_2$ be numerical semigroups in $\N$. Let $S$ be a subsemigroup
contained in $S_1 \cap S_2$. Then $S_1+S_2$ is isomorphic to $\Fibersum{S_1}{S_2}{S}$.
\end{cor}

\begin{example}
Let $S_1=\langle 2,5\rangle$ and $S_2=\langle 3,4,5\rangle$ be numerical semigroups in $\N$.
Let $S$ be a subsemigroup in $S_1 \cap S_2 =\langle 4, 5, 6, 7 \rangle$. Then
$\Fibersum{S_1}{S_2}{S}=S_1+ S_2=\langle 2,3\rangle$.
\end{example}

Without the conditions on the ranks of affine semigroups,
Proposition~\ref{prop:58908} does not hold in general.

\begin{example}\label{ex:51661}
Affine semigroups  $S_1=\langle(1,0),(1,1)\rangle$ and $S_2=\langle(1,1),(0,1)\rangle$ in $\N^2$ have rank 2. 
Their intersection $S = \langle (1,1) \rangle$ has rank 1. 
Using the homomorphisms $S_1\to \N^3$ and $S_2\to\N^3$ given by $(1,0)\mapsto(1,0,0)$,
$(1,1)\mapsto(0,1,0)$ and $(0,1)\mapsto(0,0,1)$, 
we have $\fibersum{S_1}{S_2}{S}\simeq\N^3$ by the universal property. 
Furthermore, $\fibersum{S_1}{S_2}{S}$ is isomorphic to $\Fibersum{S_1}{S_2}{S}$ since $\N^3$ is torsion free.
Taken in $\N^2$, the affine semigroup $S_1+S_2$ is $\N^2$, which is not isomorphic to $\Fibersum{S_1}{S_2}{S}$. 
\end{example}

Later in Section 5, we shall compare the fiber sum $\Fibersum{S_1}{S_2}{S}$ defined in this section with gluing.

\begin{lem}\label{lem:emb}
Let $S_1$, $S_2$ and $S$ be submonoids of a cancellative monoid $T$
such that $S\subset S_1\cap S_2$. The following two conditions are equivalent.
\begin{enumerate}[\rm (i)]\label{lem:60544}
\item
$\gp(S_1)\cap\gp(S_2)=\gp(S)$ in $\gp(T)$.
\item
The homomorphism $\fibersum{S_1}{S_2}{S}\to T$ obtained by the universal property
sending $\fibersum{s_1}{s_2}{}$ to $s_1+s_2$ is one-to-one.
\end{enumerate}
\end{lem}

\begin{proof}
Since the inclusion $\gp(S)\subset\gp(S_1)\cap\gp(S_2)$ always holds,  it suffices to verify $\gp(S_1)\cap\gp(S_2)\subset \gp(S)$ for condition (i). The latter inclusion means that  $s_1-t_1=t_2-s_2$ with $s_1,t_1\in S_1$ and $s_2,t_2\in S_2$ implies $s_1-t_1=t_2-s_2=b-a$
for some $a,b\in S$.
The last two equalities are equivalent to
$s_1 + a = t_1 + b$ in $S_1$ and $s_2 + b = t_2 + a$ in $S_2$, that is,
$\fibersum{s_1}{s_2}{} = \fibersum{t_1}{t_2}{}$ by definition as seen in (\ref{eq:47125}).
In other words, this is the same as asking if $s_1+s_2=t_1+t_2$ implies $\fibersum{s_1}{s_2}{} = \fibersum{t_1}{t_2}{}$.
Therefore $\gp(S_1)\cap\gp(S_2)\subset \gp(S)$ is equivalent to that $\fibersum{S_1}{S_2}{S}\to T$
given by $\fibersum{s_1}{s_2}{}\mapsto s_1+s_2$  is a one-to-one homomorphism.
\end{proof}

The second condition in Lemma~\ref{lem:emb}
implies that  $\fibersum{S_1}{S_2}{S}$ is isomorphic to $S_1 + S_2$ which
is a subsemigroup of $T$. In the case where $T$ is an affine semigroup and $S_1, S_2, S$ are affine subsemigroups of $T$,
it follows that $S_1 + S_2$ is also an affine subsemigroup since $T$ is torsion free.
This implies $\fibersum{S_1}{S_2}{S}$ is already torsion free and so is isomorphic to the fibered sum
$\Fibersum{S_1}{S_2}{S}$. 
Therefore, we have the following corollary.
\begin{cor}\label{cor:emb}
If $S_1$, $S_2$ and $S$ are affine subsemigroups of an affine semigroup $T$
such that $S\subset S_1\cap S_2$ and $\gp(S_1)\cap\gp(S_2)=\gp(S)$ in $\gp(T)$,
then $\Fibersum{S_1}{S_2}{S} \simeq \fibersum{S_1}{S_2}{S} \simeq S_1 + S_2$.
\end{cor}

\begin{example}
In $\N^3$, let  $S_1=\langle (4,0,0),(3,1,0),(2,2,0),(1,3,0)\rangle$ and
$S_2=\langle (3,3,0),(3,2,1),(3,1,2),(3,0,3)\rangle$
as considered in \cite[Example~1]{gim-srin:gswh} and \cite[Example~7]{gim-srin:smfrsrog}. Then
$\gp(S_2)=\{(a+c,a,c)\mid a,c\in\Z \text { and }  a+c \text{ is divisible by } 3\}$. Note that
$(0,4,0)=(2,2,0)+(2,2,0)-(4,0,0)$. Hence
$\gp(S_1)=\{(a,b,0)\mid a,b\in\Z \text { and } a+b \text { is divisible by } 4\}$. Elements of
$\gp(S_1)\cap\gp(S_2)$ are of the form $(a,a,0)$, where $a\in\Z$ satisfies
$3 \,|\, a$ and $2  \,|\,  a$. Hence
$\gp(S_1)\cap\gp(S_2)=\gp(S)$ in $\Z^3$, where $S=\langle(6,6,0)\rangle$.
Thus by Lemma~\ref{lem:emb}, $\fibersum{S_1}{S_2}{S}$ can be embedded in $\mathbb N^3$.
This shows that $\fibersum{S_1}{S_2}{S}$ is torsion free. Hence it is already an affine semigroup
and we have $\fibersum{S_1}{S_2}{S} =  \Fibersum{S_1}{S_2}{S} $.
In fact, from Corollary~\ref{cor:flat} and Lemma~\ref{lem:05664}, we will see that
$\kappa[S_1] \otimes_{\kappa[S]} \kappa[S_2]$ is an affine semigroup ring and that
$\fibersum{S_1}{S_2}{S}$ is the semigroup corresponding to it.
\end{example}

Next example shows that even if $\fibersum{S_1}{S_2}{S}$ does not embed into $T$ as stated in Lemma~\ref{lem:emb}, it is still possible that $\Fibersum{S_1}{S_2}{S}$ embeds into $T$.

\begin{example}\label{ex:notPositive}
In $\Z^2$, let $S_1=\langle(1,-1),(0,1)\rangle$, $S_2=\langle(1,1), (0,-1)\rangle$ and
$S= \langle(1,1), (1,-1)\rangle$. Then $\gp(S_1)=\gp(S_2)=\Z^2$
and
\[
\gp(S)=\{(a,b)\in\Z^2 \mid \text{$a+b$ is even}\}\neq\gp(S_1)\cap\gp(S_2).
\]
By Lemma \ref{lem:emb}, $\fibersum{S_1}{S_2}{S}$ does not embed into $\Z^2$.
However, by Proposition \ref{prop:58908}, $\Fibersum{S_1}{S_2}{S}=\langle(1,0),(0,1),(0,-1)\rangle$ which can be embedded into $\Z^2$.
\end{example}

With canonical embeddings from $S_1$ and $S_2$ to $\Fibersum{S_1}{S_2}{S}$,
Lemma~\ref{lem:60544} for $T=\Fibersum{S_1}{S_2}{S}$ gives rise to the following proposition.

\begin{prop}\label{cor:6964}
Let $S\subset S_1$ and $S\subset S_2$ be affine semigroups. The canonical
map $\fibersum{S_1}{S_2}{S}\to\Fibersum{S_1}{S_2}{S}$ is an isomorphism
if and only if $\gp(S_1)\cap\gp(S_2)=\gp(S)$ in $\gp(\Fibersum{S_1}{S_2}{S})$.
\end{prop}

We end this section by applying Lemma~\ref{lem:60544} to the numerical situations
in Corollary~\ref{numSemigp}. Note that a numerical semigroup in $\N$ has finite complement if
and only if its greatest common divisor is $1$ \cite[Lemma~2.1]{ros-gar:ns}.
For such a numerical semigroup, its group of differences is exactly $\Z$.
By rescaling the semigroups, it follows that if $S$ is a numerical semigroup in $\N$ with the greatest common divisor $d$,
then $dn \in S$ for all $n \gg 0$ and $\gp(S)=d\Z$.
The following lemma summarizes a few key
properties that hold for numerical semigroups but are not always true for affine semigroups of higher ranks.
In particular, the condition on groups of differences in Lemma~\ref{lem:60544} and Proposition~\ref{cor:6964} can be stated in terms of greatest common divisors in numerical situations.

\begin{lem}\label{lem:gcdsemigp}
Let $S_1$ and $S_2$ be numerical semigroups in $\N$. Then
$\gp(S_1 \cap S_2) = \gp(S_1) \cap \gp(S_2)$ always holds.
Moreover, if $S$ is a numerical semigroup contained in $S_1\cap S_2$,
then the following coditions are equivalent
\begin{enumerate}[\rm (i)]
\item $\gp(S) = \gp(S_1) \cap \gp(S_2)$.
\item $ \gcd S=\gcd(S_1 \cap S_2)$.
\end{enumerate}
\end{lem}

\begin{proof}
For $i = 1,2$, let $d_i = \gcd(S_i)$, and let $d = \lcm(d_1, d_2)$.  Then $\gp(S_i) = d_i\Z$ and $\gp(S_1) \cap \gp(S_2) = d\Z$.
Since $\gp(S_1 \cap S_2) \subseteq \gp(S_1) \cap \gp(S_2)$, $\gcd(S_1 \cap S_2)$ is divisible by $d$.  On the other hand, $nd \in S_1\cap S_2$ for all $n \gg 0$ since
$nd_1 \in S_1$ and $nd_2 \in S_2$ for all $n \gg 0$.  By choosing a large enough $n$, we have
$d=(n+1)d-nd\in\gp(S_1 \cap S_2)$ and so $d$ is divisible by $\gcd(S_1 \cap S_2)$. Hence,
$\gcd(S_1 \cap S_2)=d$, or equivalently, $\gp(S_1) \cap \gp(S_2) = \gp(S_1 \cap S_2)$.  
The equivalent conditions follows directly from this.
\end{proof}

\begin{cor}\label{numSemigp}
Let $S_1$ and $S_2$ be numerical semigroups in $\N$.
Let $S$ be a subsemigroup
contained in $S_1 \cap S_2$.  Then $\gcd S=\gcd(S_1 \cap S_2)$
if and only if $\Fibersum{S_1}{S_2}{S} = \fibersum{S_1}{S_2}{S}$.
\end{cor}
\begin{proof}
First note that, by Lemma~\ref{lem:gcdsemigp}, $\gcd S = \gcd(S_1 \cap S_2)$ is equivalent to $\gp(S) = \gp(S_1) \cap \gp(S_2)$. And by Proposition~\ref{cor:6964}, the latter is equivalent to
that the canonical homomorphism $\fibersum{S_1}{S_2}{S} \to \Fibersum{S_1}{S_2}{S}$ is an isomorphism,
which means that $\fibersum{S_1}{S_2}{S}$ is torsion free and hence
$\fibersum{S_1}{S_2}{S} = \Fibersum{S_1}{S_2}{S}$ by the construction of
$\Fibersum{S_1}{S_2}{S}$.
\end{proof}

\begin{example}\label{example:35610}
Consider the numerical semigroups $S_1=\langle 3, 10 \rangle$, $S_2=\langle 5, 6 \rangle$, and $S=\langle 6, 10 \rangle$ in $\N$.
Then $\Fibersum{S_1}{S_2}{S}=S_1+S_2 = \langle 3, 5 \rangle$
by Corollary~\ref{numericalFibersum}.  Furthermore, by Corollary~\ref{numSemigp}, $\Fibersum{S_1}{S_2}{S} \neq \fibersum{S_1}{S_2}{S}$, since $\gcd S \neq \gcd(S_1 \cap S_2)$. 
(c.f. Example~\ref{ex:notFlat}.)
\end{example}

\section{Ap\'{e}ry elements}\label{sec:Apery}

We call an affine semigroup {\em positive} if $0$ is the only invertible element.
Positive affine semigroups are exactly finitely generated monoids that
can be embedded into some $\N^d$ (see \cite[Theorem 7.3]{gril:cs}).
Numerical semigroups are always positive.

In this section, we define Ap\'ery elements extending the notion of Ap\'ery numbers from numerical semigroups.
For an algebra defined by positive affine semigroups, we characterize if it is free or flat using Ap\'ery elements
 in Theorem~\ref{prop1:freeflat}.
As seen in Example \ref{ex:notPositive}, the fibered sum of positive affine semigroups may not be positive.
 We end the section by presenting some conditions under which the positivity is preserved by taking fibered sums.

Let $S'$ be a monoid and $S$ be its submonoid. An element of $S'$ is
called {\em Ap\'{e}ry} over $S$ if it can not be written as $s+w$ for any
$w\in S'$ and $0\neq s\in S$. If $S'$ is a numerical semigroup and $0\neq n\in S'$,
an Ap\'{e}ry element of $S'$ over $n\N$ is called an Ap\'{e}ry number
over $n$ in the literature. A {\em representation} of an element of $S'$
over $S$ is an expression of the element of the form $s+w$, where
$s\in S$ and $w$ is an Ap\'{e}ry element of $S'$ over $S$.
If $S'$ is positive, all of its elements have a representation. Without positivity,
Ap\'{e}ry elements may not exist.
We write $\Apr(S'/S)$ for the set of Ap\'{e}ry elements of $S'$ over $S$.

\begin{example}
The affine semigroup $\Z$ does not have Ap\'{e}ry elements over $\N$.
\end{example}

\begin{lem}\label{lem:uniqueRepr}
If $S \subset S'$ are positive affine semigroups and $S \simeq \N$, then
every element of $S'$ has a unique representation over $S$.
\end{lem}
\begin{proof}
The proof for the case of numerical semigroups \cite[Corollary 2.2]{hu-kim:nsa} carries over
to this lemma.
More precisely, let $S = \langle s \rangle$ and let $w_1, w_2 \in \Apr(S'/S)$.  Suppose
\[
\ell_1 s + w_1 = \ell_2 s + w_2
\]
in $S'$ with $\ell_1, \ell_2 \in \N$.  Without loss of generality, we may assume
$\ell_1 \geq \ell_2$.  Note that
$(\ell_1 - \ell_2)s \in S$ since $\ell_1 - \ell_2 \in \N$.  Hence, we have
\[
(\ell_1 - \ell_2)s + \ell_2 s + w_1 = \ell_2 s + w_2
\]
which implies
\[
(\ell_1 - \ell_2)s + w_1 = w_2,
\]
for $S'$ is a cancellative monoid.
However, $w_2 \in \Apr(S'/S)$, so $(\ell_1 - \ell_2)s = 0$.  This implies $w_1 = w_2$ and $\ell_1 s = \ell_2 s$.
\end{proof}

Unique representation is of special interest because we will see later that in the case of affine semigroup algebras, this is equivalent to flatness as well as to freeness.

\begin{lem}
Let $S\to S_1$ and $S\to S_2$ be embeddings of affine semigroups.
If $S_2$ is positive and every element of $S_2$ has a unique representation over $S$, then every
element of $\fibersum{S_1}{S_2}{S}$ has a unique representation over $S_1$.
\end{lem}

\begin{proof}
Consider $\fibersum{s_1}{s_2}{}\in \fibersum{S_1}{S_2}{S}$. If $s_2=a+w$ for some
$a\in S$ and $w\in S_2$, then $\fibersum{s_1}{s_2}{}=(s_1+a)\oplus w$.
Since $S_2$ is positive, every element of $\fibersum{S_1}{S_2}{S}$ can be
written as $s\oplus w$ for some $s\in S_1$ and $w\in\Apr(S_2/S)$.
We claim that $\Apr( (\fibersum{S_1}{S_2}{S})/S_1)$ consists of elements in the form
of $\fibersum{0}{w}{}$ with $w \in \Apr(S_2/S)$.
If $\fibersum{t_1 }{w_1}{}=\fibersum{t_1'}{w_2}{}$ with $ t_1, t_1' \in S_1$ and
$w_1,w_2\in\Apr(S_2/S)$, then $t_1'+a=t_1+b$ in $S_1$ and $w_2+ b=w_1+ a$ in $S_2$ for some $a,b\in S$.
Since every element of $S_2$ has a unique representation over $S$,
we must have $w_1=w_2$ and $a=b$, which imply $t_1 = t_1'$. As a consequence,
$\fibersum{0}{w}{}\in\Apr((\fibersum{S_1}{S_2}{S})/S_1)$ if and only if $w\in\Apr(S_2/S)$.
Furthermore, the above argument also proves that
every element of $\fibersum{S_1}{S_2}{S}$ has a unique
representation over $S_1$.
\end{proof}

\begin{lem}{\label{liujc}}
Let $S \to S_1$ and $S \to S_2$ be embeddings of positive cancellative monoids.  Suppose $S$ is numerical.  If $\fibersum{s_1}{s_2}{} = 0$ in $\fibersum{S_1}{S_2}{S}$, then $s_i = 0$ in $S_i$ for each $i$.
\end{lem}
\begin{proof}
Since $s_1 \oplus s_2 = 0 \oplus 0$, there exist $a,b \in S$ such that
\[
\left\{\begin{array}{ll} s_1 + a = 0 + b & \mbox{ in } S_1 \\ s_2 + b = 0 + a & \mbox{ in } S_2. \end{array} \right.
\]
Let $\varphi : S \to \N$ be an embedding and let $k = \varphi(a)$ and $\ell = \varphi(b)$.  Without loss of generality, we assume that $k \geq \ell$ and let $m:=k-\ell$.
Note that $\ell a = k b$ because $\varphi(\ell a) = \ell \varphi(a) = \ell k = k \ell = k \varphi(b) = \varphi(k b)$ and $\varphi$ is one-to-one.  Moreover, in $S_1$,
\[
ma + ks_1 + \ell a = (m+\ell)a + k s_1 = ka + ks_1 = k(a+s_1) = kb = \ell a .
\]
Since $S_1$ is cancellative, $ma + ks_1 = 0$.  Since $S_1$ is positive, $ma = ks_1 = 0$ and so $k = 0$ or $s_1 = 0$.

If $k = 0$, then $a = 0$ and so $s_2 + b = 0$.  Since $S_2$ is positive, $s_2 = b = 0$ in $S_2$.  This implies $s_1 + a = 0$ in $S_1$.  By the positivity of $S_1$, $s_1 = 0$.
If $s_1 = 0$, then $a = b$ in $S_1$.  Since $s_2 + b = a$ in $S_2$ and $S_2$ is cancellative, $s_2 = 0$.
\end{proof}

\begin{prop}{\label{positive}}
For embeddings $S\to S_1$ and $S\to S_2$ of positive cancellative monoids,
the fibered sum $\fibersum{S_1}{S_2}{S}$ is positive under either of
the following conditions.
\begin{itemize}
\item[(i)] Every element of $S_2$ has a unique representation over~$S$.
\item[(ii)] $S$ is numerical.
\end{itemize}
\end{prop}
\begin{proof}
Suppose every
element of $S_2$ has a unique representation over $S$.
Assume that
$(\fibersum{s_1}{s_2}{})+(\fibersum{t_1}{t_2}{})=0$.
Then there exist $a,b\in S$ such that $s_1+t_1+a=b$ in $S_1$ and $s_2+t_2+b=a$ in $S_2$.
Since every element of $S_2$ has a unique representation over $S$, the element
$s_2+t_2$ has to be in $S$.  Hence, the equality $(s_2+t_2)+b=a$ holds in $S$ and so also in $S_1$.  Working on $S_1$, we have
\[
(s_1+t_1+a)+(s_2+t_2)=b+(s_2+t_2)=a.
\]
We may cancel out $a$ and obtain
$s_1+t_1+(s_2+t_2)=0$. By positivity of $S_1$, $s_1=t_1=s_2+t_2=0$.
By positivity of $S_2$, $s_2=t_2=0$.
We conclude that $\fibersum{S_1}{S_2}{S}$ is positive.

Next, assume that $S$ is numerical.  Suppose $\fibersum{s_1}{s_2}{}, \fibersum{t_1}{t_2}{} \in \fibersum{S_1}{S_2}{S}$ such that
$(\fibersum{s_1}{s_2}{}) + (\fibersum{t_1}{t_2}{}) = 0$ in $\fibersum{S_1}{S_2}{S}$.  By Lemma~\ref{liujc}, we have $s_i+t_i = 0$ in $S_i$ for $i=1,\, 2$.  By the positivity of $S_i$, this implies that $s_i = t_i = 0$.  Therefore, $\fibersum{s_1}{s_2}{} = \fibersum{t_1}{t_2}{} = 0$.  This shows that $\fibersum{S_1}{S_2}{S}$ is positive.
\end{proof}

The positivity of $\fibersum{S_1}{S_2}{S}$ gives the positivity of $\Fibersum{S_1}{S_2}{S}$.  More precisely,
suppose $\Fibersum{s_1}{s_2}{}, \Fibersum{t_1}{t_2}{} \in \Fibersum{S_1}{S_2}{S}$ such that
$(\Fibersum{s_1}{s_2}{}) + (\Fibersum{t_1}{t_2}{}) = 0$ in $\Fibersum{S_1}{S_2}{S}$.  Then there exists a positive integer $n$ such that
$n \big( (\fibersum{s_1}{s_2}{}) + (\fibersum{t_1}{t_2}{}) \big) = 0$ and so $n (\fibersum{s_1}{s_2}{}) + n(\fibersum{t_1}{t_2}{}) = 0$ in $\fibersum{S_1}{S_2}{S}$.  Since $\fibersum{S_1}{S_2}{S}$ is positive, this implies that $n(\fibersum{s_1}{s_2}{}) = n(\fibersum{t_1}{t_2}{}) = 0$, and so $\Fibersum{s_1}{s_2}{} = \Fibersum{t_1}{t_2}{} = 0$ and this shows that $\Fibersum{S_1}{S_2}{S}$ is positive.  With this observation and Proposition~\ref{positive}, we have the following result.

\begin{cor}\label{cor:positive}
For embeddings $S\to S_1$ and $S\to S_2$ of positive affine semigroups,
the fibered sum $\Fibersum{S_1}{S_2}{S}$ is positive under either of
the following
conditions.
\begin{itemize}
\item[(i)] Every element of $S_2$ has a unique representation over~$S$.
\item[(ii)] $S$ is numerical.
\end{itemize}
\end{cor}

%%%%%%%%%%%%%%%%%%%%%%%%%%%%%%%%%%%%%%%%%

\section{Algebras}\label{sec:algebra}

%%%%%%%%%%%%%%%%%%%%%%%%%%%%%%%%%%%%%%%%%

For the remaining paper, $\kappa$ is a field. Let $S$ be a commutative monoid. The monoid ring
$\kappa[\bu^S]$ in the symbol $\bu$ is the set of formal finite sums
$\sum_{s\in S} a_s\bu^s$,
where $a_s\in\kappa$, with term-wise addition and the multiplication given by
$(a_s\bu^s)(a_t\bu^t)=a_sa_t\bu^{s+t}$. If the commutative monoid is an affine
semigroup, the symbol $\bu^s$ can be interpreted as a monomial. By definition, an
affine semigroup $S$ can be embedded into $\Z^d$. With the embedding,
an element of $S$ is realized as a $d$-tuple $(i_1,\ldots,i_d)$ of integers. The
{\em affine semigroup ring} $R:=\kappa[\bu^S]$ is the
subring of the Laurent polynomial ring $\kappa[\bu,\bu^{-1}]$  in the variables $\bu$ generated by
$u_1^{i_1}\cdots u_d^{i_d}$, where $(i_1,\ldots,i_d)\in S$. An affine semigroup
can also be embedded to $\Q^d$. Monomials with rational exponents naturally occur
in comparison of affine semigroup rings with various embeddings.

\begin{example}\label{ex:089685}
The affine semigroup $\langle(2,0), (1,1), (0,2)\rangle\subset\N^2$ also
embeds into $\Q^2$ through $(1,0)\mapsto(1/2,0)$ and $(0,1)\mapsto(1/2,1)$.
These two embeddings give rise to an isomorphism
\[
\kappa[u_1^2,u_1u_2,u_2^2]\simeq
\kappa[x,xy,xy^2]
\]
of $\kappa$-algebras, where $u_1\mapsto x^{1/2}$ and $u_2\mapsto x^{1/2}y$.
The affine semigroup also embeds into $\Q^3$ through $(1,0)\mapsto(1,1/2,0)$ and $(0,1)\mapsto(0,1/2,1)$. 
These two embeddings give rise to an isomorphism
\[
\kappa[u_1^2,u_1u_2,u_2^2]\simeq
\kappa[x^2y,xyz,yz^2]
\]
of $\kappa$-algebras, where $u_1\mapsto xy^{1/2}$ and $u_2\mapsto y^{1/2}z$.
Another embedding into
$\Q^3$ through $(1,0)\mapsto(1/2,1,0)$ and $(0,1)\mapsto(1/2,0,1)$
gives rise to an isomorphism
\[
\kappa[u_1^2,u_1u_2,u_2^2]\simeq
\kappa[xy^2,xyz,xz^2]
\]
of $\kappa$-algebras, where $u_1\mapsto x^{1/2}y$ and $u_2\mapsto x^{1/2}z$.
\end{example}

The affine semigroup rings in the following example appeared in \cite[Example 7]{gim-srin:smfrsrog} and \cite[Example 1]{gim-srin:gswh}.
In terms of rational exponents, we see that these rings are isomorphic.

\begin{example}\label{ex:2copies}
The affine semigroup $\langle(3,0), (2,1), (1,2), (0,3)\rangle\subset\N^2$ embeds into
$\N^3$ through $(1,0)\mapsto(1,1,0)$ and $(0,1)\mapsto(1,0,1)$. The embedding
gives rise to an isomorphism
\[
\kappa[u_1^3,u_1^2u_2,u_1u_2^2,u_2^3]\simeq
\kappa[x^3y^3,x^3y^2z,x^3yz^2,x^3z^3]
\]
of $\kappa$-algebras, where $u_1\mapsto xy$ and $u_2\mapsto xz$.
The affine semigroup also embeds into $\Q^2$ through $(1,0)\mapsto(4/3,0)$ and $(0,1)\mapsto(1/3,1)$. 
The latter embedding gives rise to an isomorphism
\begin{equation}\label{eq:87541}
\kappa[u_1^3,u_1^2u_2,u_1u_2^2,u_2^3]\simeq
\kappa[x^4,x^3y,x^2y^2,xy^3]
\end{equation}
of $\kappa$-algebras, where $u_1\mapsto x^{4/3}$ and $u_2\mapsto x^{1/3}y$.
\end{example}

From the above perspective, an affine semigroup ring $R=\kappa[\bu^S]$ is not only a ring, but a ring together with
an embedding $R\to\kappa[\bu,\bu^{-1}]$ indicated by the variables $\bu$.
We use the notation $S=\log_\bu R$ for the correspondence between the semigroup
$S$ and the commutative ring $R$. As seen in Examples~\ref{ex:089685} and \ref{ex:2copies},
an affine semigroup ring may occur in different guises from various embeddings of the semigroup.
For an affine semigroup ring $R$ in a set of variables
$\bu$ and in another set of variables $\bv$, there is an isomorphism
$\log_\bv\bu\colon\log_\bu R\to\log_\bv R$, which can be
described by an invertible matrix. For instance,
\[
\log_\bv\bu=\begin{pmatrix}4/3& 1/3\\ 0 &1\end{pmatrix}
\]
indicates the change of embeddings that gives rise to the isomorphism~(\ref{eq:87541}).
If $\log_\bu R$ and $\log_\bv R$ are subsets of $\Q$, the isomorphism
$\log_\bv\bu$ is simply the multiplication by the number $s\in\Q$ given by
$\bv=\bu^s$. If $s = \log_{\bu} {\bv} \in\N$, we say that the embedding of $R$  in
$\kappa[\bu]$ is {\em finer} than its embedding in $\kappa[\bv]$.

If $S$ is an affine subsemigroup of another affine semigroup $S'$, we may consider them
as subsemigroups of some $\Z^d$. The semigroup ring $R':=\kappa[\bu^{S'}]$ is an
algebra over the coefficient ring $R:=\kappa[\bu^S]$.
We use $R'/R$ to denote the {\em affine semigroup algebra} $R'$ over $R$.
An affine semigroup ring $\kappa[\bu^S]$ is {\em positive} if $\bu^0$ is the only invertible monomial;
or equivalently, the affine semigroup $S$ is positive.
We call an affine semigroup algebra $R'/R$ positive, if $R'$ is positive.
The coefficient ring of a positive affine semigroup algebra is also positive.

Let $R'/R$ be an affine semigroup algebra in the variables $\bu$.
A monomial  $\bu^w\in R'$ is {\em Ap\'{e}ry} over $R$ if $w$ is Ap\'{e}ry over the semigroup
$\log_\bu R$. In other words, $\bu^w$ is {\em Ap\'{e}ry} over $R$ if
it can not be written as $\bu^s\bu^{w_0}$ for any $\bu^{w_0}\in R'$ and
$1\neq \bu^s\in R$. The set of Ap\'{e}ry monomials of $R'$ over $R$ is denoted by $\Apr(R'/R)$.
The notion of representations in affine semigroups also has an exponential counterpart.
To be precise, a {\em representation} of a monomial  $\bu^{s'}\in R'$
over $R$ is an expression $\bu^{s'}=\bu^s\bu^w$, where $\bu^s\in R$ and $\bu^w\in \Apr(R'/R)$.
If $R'/R$ is positive, any monomial in $R'$ has a representation over $R$.

\begin{thm}\label{prop1:freeflat}
Let $R'/R$ be a positive affine semigroup algebra in the variables $\bu$.
The following conditions are equivalent.
\begin{enumerate}[\em (i)]
\item Every monomial in $R'$ has a unique representation over $R$.
\item $R'$ is free over $R$.
\item $R'$ is flat over $R$.
\end{enumerate}
\end{thm}
\begin{proof}First note that because $R$ and $R'$ are both subrings of $\kappa[\bu]$ generated by monomials, 
if $f$ is in  $R$ (resp. $R'$), then every monomial appearing in $f$ is contained in $R$ (resp. $R'$).  Hence, if every monomial in $R'$ has a unique representation, then $R'$ is a
free module over $R$ with $\Apr(R'/R)$ as a basis.  More precisely, suppose there exist distinct $\bu^{w_1}, \bu^{w_2}, \ldots, \bu^{w_n} \in \Apr(R'/R)$ and nonzero $f_1, f_2, \ldots, f_n \in R$ such that $f_1 \bu^{w_1} + f_2 \bu^{w_2} + \cdots + f_n \bu^{w_n} = 0$.   Since $f_1$ is nonzero, let $\bu^{s_1}$ be a monomial appearing in $f_1$. Then the monomial $\bu^{s_1}\bu^{w_1}$ must appear in $f_i\bu^{w_i}$ for some $i$, and so $\bu^{s_i}\bu^{w_i} = \bu^{s_1}\bu^{w_1}$ where $\bu^{s_i}$ is a monomial appearing in $f_i$.  Since $\bu^{w_1}$ and $\bu^{w_i}$ are distinct,
the equality $\bu^{s_i}\bu^{w_i} = \bu^{s_1}\bu^{w_1}$ contradicts the assumption that every monomial in $R'$ has a unique representation over $R$.
This proves that (i) implies (ii).

Since free modules are
flat, it remains to show that every monomial in $R'$ has a unique representation over $R$ if $R'$ is flat over $R$.
Suppose there exist distinct $\bu^{w_1}, \bu^{w_2} \in \Apr(R'/R)$ and $\bu^{s_1}, \bu^{s_2} \in R$ such that
\[
\bu^{s_1}\bu^{w_1} = \bu^{s_2} \bu^{w_2}, \quad \mbox{i.e., } \bu^{s_1}\bu^{w_1} - \bu^{s_2} \bu^{w_2} = 0.
\]
Since $R'$ is flat over $R$, by \cite[Theorem 7.6]{mats:crt}, there exist matrices
\[
\mathcal{M} = \left[ \begin{array}{ccc} \smallskip f_{11} & \cdots & f_{1\ell} \\ f_{21} & \cdots & f_{2\ell} \end{array}\right] \qquad \mbox{and} \qquad \left[ \begin{array}{c} g_1 \\ \vdots \\ g_\ell \end{array}\right]
\]
with $f_{ij} \in R$ and $g_i \in R'$ such that
\[
\left[\begin{array}{cc} \bu^{s_1} & \bu^{s_2} \end{array}\right] \mathcal{M} = 0 \qquad \mbox{and} \qquad
\mathcal{M}\left[ \begin{array}{c} g_1 \\ \vdots \\ g_\ell \end{array}\right] = \left[\begin{array}{c} \bu^{w_1} \\ \bu^{w_2} \end{array}\right].
\]
Since $f_{11}g_1 + \cdots + f_{1\ell} g_\ell = \bu^{w_1}$, the monomial $\bu^{w_1}$ must appear in $f_{1i}g_i$ for some $i$, and so we have $\bu^{w_1}=\bu^t\bu^{t'}$, where $\bu^t$ and $\bu^{t'}$ are monomials appearing in $f_{1i}$ and $g_i$, respectively.  
That $f_{1i} \in R$ and $g_i \in R'$ implies $\bu^t \in R$ and $\bu^{t'} \in R'$.  Since $\bu^{w_1} \in \Apr(R'/R)$, $\bu^t = 1$.  Moreover, since $\bu^{s_1}f_{1i} + \bu^{s_2}f_{2i} = 0$ and $\bu^t$ is a monomial appearing in $f_{1i}$, the monomial $\bu^{s_1} \bu^t = \bu^{s_1}$ appears in $\bu^{s_1}f_{1i}$ and so $\bu^{s_1}$ appears in $\bu^{s_2}f_{2i}$.  
This implies $\bu^{s_1} = \bu^{s_2}\bu^{r}$ for some monomial $\bu^{r}$ appearing in $f_{2i}$. 
Notice that $\bu^{r}$ must be in $R$ because $f_{2i} \in R$. Therefore
\[
\bu^{s_1}\bu^{w_1} = \bu^{s_2}\bu^{w_2} \implies \bu^{s_2}\bu^{r}\bu^{w_1} = \bu^{s_2}\bu^{w_2} \implies \bu^{r}\bu^{w_1} = \bu^{w_2} .
\]
Since $\bu^{w_2} \in \Apr(R'/R)$ and $\bu^{r} \in R$, $\bu^{r} = 1$ and so $\bu^{w_1} = \bu^{w_2}$, 
which contradicts the assumption that $\bu^{w_1}$ and $\bu^{w_2}$ are distinct.
\end{proof}

The following corollary is an immediate consequence of Lemma~\ref{lem:uniqueRepr} and Theorem~\ref{prop1:freeflat}.
\begin{cor}\label{cor:flat}
Let $\kappa[\bu^{S'}]$ be a positive affine semigroup ring and $s\in {S'}$. Then $\kappa[\bu^{S'}]$ is flat
over $\kappa[\bu^s]$.
\end{cor}

Consider affine semigroup algebras $R_1/R$ in variables $\bu$ and $R_2/R$
in variables $\bv$. Denote $S_1=\log_\bu R_1$, $S_2=\log_\bv R_2$ and
$S=\log_\bu R\simeq\log_\bv R$.  In previous sections, we construct two cancellative monoids $\fibersum{S_1}{S_2}{S}$ and $\Fibersum{S_1}{S_2}{S}$. The former is finitely generated  but not necessarily torsion free. 
The latter is the quotient of the former by the equivalence relation $\sim_{\tor}$ and thus is an affine semigroup.  
In particular, using the fact that $\Fibersum{S_1}{S_2}{S}$ is a fibered sum of $S_1$ and $S_2$ over $S$ in the category of affine semigroups, it can be checked directly that $\kappa[\bw^{\Fibersum{S_1}{S_2}{S}}]$ over $R$ satisfies the universal property,
 and thus is a fibered sum of $R_1/R$ and $R_2/R$ in the category of affine semigroup algebras.  
 We call the ring $\kappa[\bw^{\Fibersum{S_1}{S_2}{S}}]$ a fibered sum of the affine semigroup rings $R_1$ and $R_2$ over $R$.
In order to understand $\kappa[\bw^{\fibersum{S_1}{S_2}{S}}]$ and $\kappa[\bw^{\Fibersum{S_1}{S_2}{S}}]$ better,
it is desirable to compare them with other rings that naturally arise from $R_1$ and $R_2$, 
such as their tensor product over $R$.
The following lemma proves that if one of the algebras is positive and flat, then $\kappa[\bw^{\fibersum{S_1}{S_2}{S}}]$ is isomorphic to $R_1\otimes_R R_2$.

\begin{lem}\label{lem:05664}
Consider affine semigroup algebras $R_1/R$ in variables $\bu$ and $R_2/R$
in variables $\bv$. Denote $S_1=\log_\bu R_1$ and $S_2=\log_\bv R_2$. Let $S$ be a semigroup isomorphic to both
$\log_\bu R$ and $\log_\bv R$. If $R_2$ is positive and flat over $R$, then
$\kappa[\bw^{\fibersum{S_1}{S_2}{S}}]\simeq R_1\otimes_R R_2$ as $R$-algebras.
\end{lem}
\begin{proof}
Since $R_2$ is positive and flat over $R$, elements of $S_2$ can be written uniquely as $s+w$ for some $s\in S$ and
$w\in\Apr(S_2/S)$. Hence elements of $\fibersum{S_1}{S_2}{S}$ can be written uniquely as
$s_1\oplus w$ for some $s_1\in S_1$ and $w\in\Apr(S_2/S)$. We may
define a $\kappa$-linear map $\kappa[\bw^{\fibersum{S_1}{S_2}{S}}]\to R_1\otimes_R R_2$ by
$\bw^{s_1\oplus w}\mapsto\bu^{s_1}\otimes\bv^w$, which is indeed a ring homomorphism. The universal property of the tensor product
provides an $R$-linear map $R_1\otimes_R R_2\to\kappa[\bw^{\fibersum{S_1}{S_2}{S}}]$ satisfying
$\bu^{s_1}\otimes\bv^{s_2}\mapsto\bw^{\fibersum{s_1}{s_2}{}}$. These two maps are inverse to each other and so they are indeed isomorphisms of $R$-algebras.
\end{proof}

Observe that for Lemma~\ref{lem:05664}, in addition to the assumption that $R_2$ is positive and flat over $R$,
if the groups of differences of these semigroups are nicely related,
namely $\gp(S_1)\cap \gp(S_2) = \gp(S)$ in $\gp(\Fibersum{S_1}{S_2}{S})$,
then by Proposition~\ref{cor:6964}, $\fibersum{S_1}{S_2}{S}$ and $\Fibersum{S_1}{S_2}{S}$ are isomorphic and hence
$\kappa[\bw^{\Fibersum{S_1}{S_2}{S}}]$ is also isomorphic to $R_1\otimes_R R_2$.  In other words,
Lemma~\ref{lem:05664} and Proposition~\ref{cor:6964} show the following theorem.

\begin{thm}\label{thm:Fibersumetensor}
With notation as in Lemma~\ref{lem:05664},
if $R_2$ is positive and flat over $R$ and if $\gp(S_1)\cap \gp(S_2) = \gp(S)$ in $\gp(\Fibersum{S_1}{S_2}{S})$,
then
$\kappa[\bw^{\Fibersum{S_1}{S_2}{S}}]\simeq R_1\otimes_R R_2$ as $R$-algebras.
\end{thm}

In the category of commutative rings, nice properties of $R_1/R$ in the diagram
\[
\xymatrix{
R_1 \ar[r]& \tensor{R_1}{R_2}{R} \\
R  \ar[u]\ar[r]   &  R_2  \ar[u]
}
\]
pass to the algebra $R_1 \otimes_R R_2$ over $R_2$ under base change from $R$ to $R_2$.
In addition if $R_2$ is positive and flat over $R$, then by Lemma~\ref{lem:05664}, $\kappa[\bw^{\fibersum{S_1}{S_2}{S}}]$ is isomorphic to $R_1 \otimes_R R_2$, so the algebra $\kappa[\bw^{\fibersum{S_1}{S_2}{S}}]$ over $R_2$ inherits some nice properties of $R_1$ over $R$. With  notation and  assumptions as in Lemma~\ref{lem:05664},
note that if $S_1$ is also positive, then the affine semigroup $\Fibersum{S_1}{S_2}{S} \simeq (\fibersum{S_1}{S_2}{S})/\sim_{\tor}$ is positive by Corollary~\ref{cor:positive}(i). This leads us to adapt the following notation.

\begin{notation}\label{notation:otimes+}
With notation as in Lemma~\ref{lem:05664}, assuming that 
both $S_1$ and $S_2$ are positive and that $R_2$ is flat over $R$, 
we use $\Tensor{R_1}{R_2}{R}$ to denote the positive affine semigroup ring $\kappa[\bw^{\Fibersum{S_1}{S_2}{S}}]$ which is
the fibered sum of $R_1$ and $R_2$ over $R$.
\end{notation}

Furthermore, in the proof of Lemma~\ref{lem:05664}, we obtain an isomorphism
$R_1\otimes_R R_2 \to \kappa[\bw^{\fibersum{S_1}{S_2}{S}}]$ given by
$\bu^{s_1}\otimes\bv^{s_2} \mapsto \bw^{\fibersum{s_1}{s_2}{}}$.  Hence, it is natural for us to identify the monomial $\bw^{\fibersum{s_1}{s_2}{}}$ in $\kappa[\bw^{\fibersum{S_1}{S_2}{S}}]$ with the element $\bu^{s_1}\otimes\bv^{s_2}$ in $R_1\otimes_R R_2$.  Since Notation~\ref{notation:otimes+} introduces $\Tensor{R_1}{R_2}{R}$ for  $\kappa[\bw^{\Fibersum{S_1}{S_2}{S}}]$, it is natural to denote the monomial $\bw^{\Fibersum{s_1}{s_2}{}}$ in $\kappa[\bw^{\Fibersum{S_1}{S_2}{S}}]$ as $\Tensor{\bu^{s_1}}{\bv^{s_2}}{}$.

\begin{defn}\label{def:BaseChange}
Let $R_1/R$ and $R_2/R$ be positive affine semigroup algebras over $R$.
Assume that $R_2/R$ is flat. We say
that the positive affine semigroup algebra $\Tensor{R_1}{R_2}{R}$ over $R_2$ is the
{\em flat base change of $R_1/R$ by $R_2/R$}.
\end{defn}

Flat base change can be illustrated in the following commutative diagram under the condition that $R_2$ is flat over $R$:
\[
\xymatrix{
R_1 \ar[r]& \Tensor{R_1}{R_2}{R} \\
R  \ar[u]\ar[r]   &  R_2 \ar[u]
}
\]
We are interested in knowing what property captured in the vertical algebra $R_1/R$ passes to the other vertical algebra via flat base change.

\begin{example}\label{ex:algebra6Nvs12N}
Consider the affine semigroup rings $R = \kappa[X^{6n}]$, $R_1 = \kappa[X^3]$, $R_2 = \kappa[X^{2}]$.  First note that
by Proposition~\ref{prop:58908}, $\Fibersum{3\N}{2\N}{6n\N} \simeq 3\N + 2\N$, so $\Tensor{R_1}{R_2}{R} \simeq \kappa[X^{2},X^{3}]$.  Moreover,
by Corollary~\ref{cor:flat}, $R_2/R$ is flat.  We investigate the flat base change of $R_1/R$ by $R_2/R$ as illustrated below.
\[
\xymatrix{
R_1 = \kappa[X^3] \ar[r]& \Tensor{R_1}{R_2}{R} \simeq \kappa[X^{2},X^{3}] \\
R = \kappa[X^{6n}]  \ar[u]\ar[r]   &  R_2 = \kappa[X^{2}] \ar[u]
}
\]
If $n = 1$, then the Ap\'{e}ry monomials of $R_1/R$, namely $1$ and $X^3$, pass to the Ap\'{e}ry monomials
$\Tensor{1}{1}{}$ and $\Tensor{X^3}{1}{}$ of $(\Tensor{R_1}{R_2}{R})/R_2$.
However, if $n>1$, the Ap\'{e}ry monomial $X^6$ of $R_1/R$
passes to $\Tensor{X^6}{1}{}$
which is no longer Ap\'{e}ry in $(\Tensor{R_1}{R_2}{R})/R_2$ since $\Tensor{X^6}{1}{} = (X^2)^3 (\Tensor{1}{1}{})$.
By straightforward computations,
$\Tensor{1}{1}{}$ and $\Tensor{X^3}{1}{}$ remain to be the only Ap\'ery monomials.
Precisely, we notice that  $6n \N \subset 6\N = 2 \N \cap 3 \N $ for all $n \geq 1$ and that
the fiber sum of $R_1$ and $R_2$ over $\kappa[X^{6n}]$ is isomorphic to $\Tensor{R_1}{R_2}{\kappa[X^{6}]}$ for all $n$.
(See also the discussion following Example~\ref{ex:6Nvs12N}.) 
\end{example}

In order to define the flat base change in the category of affine semigroup rings,
Definition \ref{def:BaseChange} assumes that $R_2/R$ is flat.
Even assuming further that $R_1/R$ is flat,
flatness does not always pass to the algebra $\Tensor{R_1}{R_2}{R}$ over $R_2$
nor over $R$. The next example demonstrates these phenomena.

\begin{example}\label{ex:notFlat}
By \cite[Proposition 2.5]{hu-kim:nsa}, adding the square root of a monomial to $R=\kappa[X^3,X^5]$, we obtain a flat algebra.
For instance, $R_1=\kappa[X^{3/2},X^5]$ and $R_2=\kappa[X^3,X^{5/2}]$ are flat over $R$.
By Corollary~\ref{numericalFibersum}, $\Tensor{R_1}{R_2}{R} \simeq \kappa[X^{3/2},X^{5/2}]$.
The Ap\'ery monomials of $R_1/R$ are $1$ and $X^{3/2}$. In this example, they are exactly the  Ap\'ery monomials of $\Tensor{R_1}{R_2}{R}$ over $R_2$ via the isomorphism just mentioned. However  $1$ and $X^{3/2}$ are not linearly independent over $R_2$ since $X^6X^{3/2}=X^{15/2} \cdot 1$. This shows that the flatness of $R_1/R$ does not pass through the flat base change.

Next, we observe that the Ap\'ery monomials of $\Tensor{R_1}{R_2}{R}$ over $R$ are $1$, $X^{3/2}$, $X^{5/2}$ and
$X^{3/2} X^{5/2}$. Re-arranging the above equation, one sees distinct representations
$X^6X^{3/2}=X^5X^{5/2}$. This shows that, by Theorem~\ref{prop1:freeflat},
the algebra $\kappa[X^{3/2},X^{5/2}]$ is not flat over $R$.

The tensor product $\tensor{R_1}{R_2}{R}$ is free over $R_2$ and $R$ respectively due to the flatness of $R_2$  over $R$.
Thus $\tensor{R_1}{R_2}{R}$ and $\Tensor{R_1}{R_2}{R}$ are not isomorphic. 
In fact, $\tensor{R_1}{R_2}{R} \simeq R_2[t]/\langle t^2 - X^3 \rangle$ is not an integral domain,
because $t^{10} - X^{15}$ in the ideal $\langle t^2 - X^3 \rangle$ can be factored as $(t^5 + X^{15/2}) (t^5 - X^{15/2})$ in $R_2[t]$ and neither factor is contained in $\langle t^2 - X^3 \rangle$.
This example in semigroup algebras corresponds to the results of its exponent counterpart in Example~\ref{example:35610}.
\end{example}

\begin{example}\label{ex:folding}
Consider $R_1 = \kappa[X, XY]$ and $R_2= \kappa[XY,Y]$. Let $R :=R_1\cap R_2=\kappa[ XY]$.
Since $R_1$ and $R_2$ are polynomial rings over $R$, they are
flat over $R$.
By Lemma~\ref{lem:05664}, the ring generated by the fibered sum of the corresponding cancellative monoids is
\begin{equation}\label{eq:57167}
\kappa[X, XY]\otimes_{\kappa[XY]}\kappa[XY,Y] \simeq \kappa[X,XY, Z]
\end{equation}
which is already an affine semigroup ring. Thus, it is the fibered sum $\Tensor{R_1}{R_2}{R}$ of $R_1$ and $R_2$ over $R$.
\end{example}

Example~\ref{ex:folding} shows that the fibered sum may lead to an increase of dimension of the rings.
In particular, if $R_1$ or $R_2$ is flat over $R$, then by Lemma~\ref{lem:05664}
the fibered sum is the tensor product or a quotient of the tensor product.
In Section~\ref{sec:gluing}, we review gluing from \cite{ros:psN} which coincides as a special case of fibered sums.
For example, with a change of embeddings,
Example~\ref{ex:folding} can be viewed as the gluing of two polynomial rings.
Geometrically, these are the coordinate rings of two affine planes.
In this case, the fibered sum identifies an affine line from each plane -- as if the fibered sum glues the two planes along a line, and naturally creates a three dimensional affine space.
In terms of semigroups, the semigroups in Example~\ref{ex:folding} are generated by
the monomials in the cones spanned respectively by $X, XY$ and by $XY, Y$.
The latter one is isomorphic to the cone spanned by $XY$ and $Z$ in a three dimensional space.
One thinks that the gluing happens when there is a room for folding up the two cones along the one cone spanned by $XY$
such that the folding creates a crease where the gluing takes place. Thus the resulting folded semigroup generates a three dimensional semigroup ring as in (\ref{eq:57167}).

Such observation reveals the advantage of considering fibered sum and this will become clear in Section~\ref{sec:gluing}.
In fact, as we will see in the discussion following
Example~\ref{ex:normalcurve2}, gluing is not intrinsic, and not every pair of semigroup rings can be glued.
In particular, by \cite[Therom~1.4]{gim-srin:gsNc},
$\kappa[X,XY]$ and $\kappa[XY,Y]$ cannot be glued under any circumstance since they are both {\em nondegenerate} in $\kappa[X,Y]$
(i.e. $\kappa[X,XY]$ and $\kappa[XY,Y]$ have the same dimension as $\kappa[X,Y]$). 
On the other hand, $ \kappa[X,XY]$ and $\kappa[XY, Z] $  can be glued since they are degenerate
in the three dimensional polynomial ring $\kappa[X,Y,Z]$, the smallest space  where they coexist.
Fibered sums always take place regardless initial embeddings,
and therefore, provide one possible intrinsic interpretation for gluing.

\section{Fibered Sum versus Gluing}\label{sec:gluing}

Let $T_1$ and $T_2$ be numerical semigroups in $\N$, both with the greatest
common divisor $1$. Let $a\in T_1$ and $b\in T_2$ be relatively prime numbers
such that $a$ (resp. $b$) is not a minimal generator of $T_1$ (resp. $T_2$).
Under these conditions on $a$ and $b$, the
numerical semigroup $bT_1+aT_2$
is called a {\em gluing} of numerical semigroups (c.f. \cite[p.\,130]{ros-gar:ns}).

The gluing is in fact a fibered sum with respect to a third numerical semigroup to be unveiled here: 
we look at this fact from the perspective of considering them as exponents in the semigroup rings.
Namely,  we consider $R_1:=\kappa[u_1^{T_1}]$ and $R_2:=\kappa[u_2^{T_2}]$
as subrings of the polynomial rings $\kappa[u_1]$ and $\kappa[u_2]$ in variables $u_1$ and $u_2$, respectively.
The aim of gluing is to equate the monomials $u_1^a$ and $u_2^b$.
In fact, if we identify $u_1^{1/b}$ and $u_2^{1/a}$ and  call it $u$,
then the new variable $u$ works for this purpose.
Furthermore,  $R_1 =\kappa[u^{bT_1}]$ and $R_2=\kappa[u^{aT_2}]$ as subsets of $\kappa[u^{bT_1+aT_2}]$
have an intersection $R:=\kappa[u^{ab}]$. In the way of
\[ u_1^a = (u_1^{1/b})^{ba} = u ^{ba} = u ^{ab} =  (u_2^{1/a})^{ab} = u_2^b, \]
$u_1^a$ and $u_2^b$ are ``glued''.
Precisely, the process of gluing identifies the subring $\kappa[u_1^a]$ of 
$\kappa[u_1^{T_1}]$ and $\kappa[u_2^b]$ of $\kappa[u_2^{T_2}]$.
If we change their embeddings all into $\kappa[u]$, then both $u_1^a$ and $u_2^b$ are mapped to $u^{ab}$.
Moreover, $R_1$ and $R_2$ now has a subring $\kappa[u^{ab}]$ in common
over which the resulting gluing
$\kappa[u^{bT_1+aT_2}]$ appears as a universal object.
Returning to the semigroup perspective,
by Corollary~\ref{numericalFibersum}, $aT_1+bT_2$ is exactly the fibered sum of $aT_1$ and $bT_2$ over
$ab \N$ as defined in Definition~\ref{def:fibersum}.
We remark that, as far as fibered sums are concerned, 
the condition stated in the above excluding $a$ and $b$ being minimal generators for gluing is not necessary.
The relatively prime condition on $a$ and $b$ assures that the gluing has  greatest common divisor $1$.
This condition is neither essential.

Our interpretation of gluing imposes vital usage of variables $\bu$ that accommodate the
algebras $R_1/R$ and $R_2/R$ as well as the ring $R_1\widetilde{\otimes}_R R_2$.
In the classical treatment of affine semigroup rings, the role of variables is not brought
out into open. For an affine semigroup $S$, the associated affine semigroup ring is often
denoted by $\kappa[S]$, without referencing variables.
However, coming to actual operations of elements in the ring,
recognition of variables becomes essential in understanding the underline ring structures.

{\bf Gluing.}
The gluing of numerical semigroup rings is generalized
to higher dimension by Rosales~\cite{ros:psN}.
Let $S$ be an affine semigroup. 
We consider the affine semigroup ring $\kappa[\bv^S]$ in multivariable $\bv$.
For $a\in\N$ and new multivariable $\bu=\bv^{1/a}$, we say that $\kappa[\bu^{aS}] = \kappa[ \bv^S]$ has an embedding in $\kappa[\bu]$ {\em finer} than its
embedding in $\kappa[\bv]$,
similarly to the numerical cases described after Example~\ref{ex:2copies}.
Let $S_1=\langle s_1,\ldots,s_n\rangle$ and $S_2=\langle t_1,\ldots,t_m\rangle$ be positive affine semigroups generated by $s_1, .. , s_n$ and $t_1,.., t_m$ in $\N^d$ respectively.
Rosales~\cite{ros:psN} defines that {\em $S_1$ and $S_2$ can be glued} if 
there exist $a,b\in\N$ such that the $\kappa$-algebra homomorphisms
\begin{eqnarray}\label{glue_varphi}
\varphi_1&:&\kappa[X_1,\ldots,X_n]\to\kappa[\bu^{aS_1}],\nonumber\\
\varphi_2&:&\kappa[Y_1,\ldots,Y_m]\to\kappa[\bu^{bS_2}],\label{eq:61627}\\
\varphi&:&\kappa[X_1,\ldots,X_n,Y_1,\ldots,Y_m]\to\kappa[\bu^{aS_1+bS_2}]\nonumber
\end{eqnarray}
given by $\varphi(X_i)=\varphi_1(X_i)=\bu^{as_i}$ and
$\varphi(Y_j)=\varphi_2(Y_j)=\bu^{bt_j}$ satisfying the condition
\begin{equation}\label{glueKernel}
\ker\varphi=\ker\varphi_1+\ker\varphi_2+
\langle X^{i_1}_1\cdots X^{i_n}_n-Y^{j_1}_1\cdots Y^{j_m}_m\rangle
\end{equation}
for some monomials $X^{i_1}_1\cdots X^{i_n}_n$ and $Y^{j_1}_1\cdots Y^{j_m}_m$.
In our perception, what Rosales means by saying that {\em $S_1$ and $S_2$ can be glued} actually means that
{\em $\kappa[\bu^{S_1}]$ and $\kappa[\bu^{S_2}]$ can be glued along a polynomial ring in one variable in some finer embedding}.

The next example illustrates the above description.

\begin{example}\label{ex:normalcurve2}
Consider the affine semigroups $S_1=\langle (2,1,0),(1,1,1),(0,1,2)\rangle$ and
$S_2=\langle (1,2,0),(1,1,1),(1,0,2)\rangle$ in $\N^3$. We use $x$, $y$, $z$ for variables in $\bf u$.
Then
$R_1=\kappa[\bu^{S_1}] = \kappa[ x^2y, xyz, yz^2]$ and
$R_2=\kappa[\bu^{S_2}] = \kappa[xy^2, xyz, xz^2]$.
In this example, $m=n=3$. By taking $a=b=1$,
$\kappa[\bu^{aS_1+bS_2}] = \kappa[ x^2y, yz^2, xyz, xy^2, xz^2]$.
The above ring homomorphisms $\varphi_1$ maps $X_1, X_2, X_3$ to $x^2y, xyz, yz$, 
and similarly for $Y_1, Y_2$, and $Y_3$. 
Lemma~\ref{lem:gluingCond} will show $\ker\varphi = \ker \varphi_1 + \ker \varphi_2+ \langle X_2-Y_2\rangle$,
which means only $xyz$ from $R_1$ and $R_2$ are identified.
Therefore, $R_1$ and $R_2$ can be glued 
along $\kappa[xyz]$
to obtain $\kappa[ x^2y, yz^2, xyz, xy^2, xz^2]$.
This is done without identifying powers of other monomials. 
\end{example}

Next, we discuss the connections between the fibered sum and the gluing just described.
Consider $R'=R_1'=R_2' = \kappa[u^2, uv, v^2]$ and the corresponding homomorphisms 
as in (\ref{eq:61627}), but denoted here by $\varphi_1'$, $\varphi_2'$, and $\varphi'$,
with $a=b=1$.
Besides elements of $\ker \varphi_1'$ and $\ker \varphi_2'$,
the kernel of $\varphi'$ is also generated by $X_1-Y_1$, $X_2-Y_2$ and $X_3-Y_3$.
The semigroup $S'$ that gives rise to $R_1'$ and $R_2'$  is generated by the columns in the matrix
\[
M=\begin{pmatrix} 2 & 1 & 0 \\ 0 & 1 & 2 \end{pmatrix}.
\]
These semigroups cannot be realized by any choices of positive integers $a$ and $b$
to reduce the generating set of $\ker \varphi'$ by adding only one binomial to $\ker \varphi_1' + \ker \varphi_2'$.
Therefore,  $R_1'$ and $R_2'$ cannot be glued.
This is a nondegenerate case in  Gimenez and Srinivasan \cite{gim-srin:gsNc} 
and is supported by Theorem~1.4 of the same paper.

On the other hand, one observes that $R_1' = R_2'$ is isomorphic to $R_1$ and $R_2$ in
Example~\ref{ex:normalcurve2} where a gluing of $R_1$ and $R_2$ takes place.
This shows that gluing is not an intrinsic definition.
We propose a remedy to resolve such situation in the following.

In Example~\ref{ex:normalcurve2},  $S_1$ and $S_2$ are indeed the semigroups in $\N^3$
generated by the columns in the matrices
\[
M_1= \begin{pmatrix} 2 & 1 & 0 \\ 1 & 1 & 1 \\ 0 & 1 & 2 \end{pmatrix}
\text{ and }
M_2 = \begin{pmatrix} 1 & 1 & 1 \\ 2 & 1 & 0 \\ 0 & 1 & 2 \end{pmatrix},
\]
respectively.
Instead of a scalar multiple as defined in the gluing, we suggest a matrix transformation on the generating set.
For instance,
let
\[
A = \begin{pmatrix} 1 & 0 \\ 1/2  & 1/2 \\ 0 & 1 \end{pmatrix} \text{ and }
B = \begin{pmatrix} 1/2 & 1/2 \\  1 & 0 \\ 0 & 1 \end{pmatrix}.
\]
Then $AM = M_1$ and $BM=M_2$.
Let $AS'$ denote the semigroup generated by the columns of the matrix $AM$, and similarly $BS'$  for $BM$.
Obviously, now we have $\kappa[ \bu^{AS'}] = \kappa[\bu^{S_1}] =R_1$ and $\kappa[ \bu^{BS'}] = \kappa[\bu^{S_2}] =R_2$, and they can be glued as shown in Example~\ref{ex:normalcurve2}.

We summarize the above comparison of the attempt of gluing $R'$ and itself to gluing $R_1$ and $R_2$.
The semigroup $S'$ is of rank $2$ and embedded in $\N^2$.
This makes $R'$ nondegenerate and so it cannot be glued to itself as stated in \cite[Therom~1.4]{gim-srin:gsNc}.
By linear transformations, the semigroup $S'$ is transferred into two semigroups $S_1$ and $S_2$ in a free group of higher rank. In this way, not only $S_1$ and $S_2$ are degenerate, but they are also separate in a way that
$\gp{(S_1)} \cap \gp{(S_2)}$ is generated by a single element belonging to $S_1\cap S_2$.
The latter condition about the intersection is equivalent to requiring
$\ker \varphi$ to acquire only one more generator in addition to
$\ker \varphi_1$ and $\ker \varphi_2$.
Such an essential condition is proved in \cite[Therom~1.4]{ros:psN}
and recovered in Theorem~\ref{glueEqCond} below using fibered sums.

We first prove a theorem that interprets gluings as fibered sums.

%\marginpar{Search to change 5.2 from prop to thm}

\begin{thm}\label{thm:glueVfibersum}
If positive affine semigroup rings $R_1$ and $R_2$ in the same variables can be
glued along a polynomial ring $R$ in one variable in some finer embedding,
then there are homomorphisms $R\to R_1$ and $R\to R_2$ such that the gluing is isomorphic to
$\tensor{R_1}{R_2}{R}$. Furthermore, $\Tensor{R_1}{R_2}{R}\simeq\tensor{R_1}{R_2}{R}$.
\end{thm}

\begin{proof}
Denote $S_1=\log_\bu R_1$ and $S_2=\log_\bu R_2$. Let
$i_1,\ldots,i_n,j_1,\cdots,j_m$ and $a,b$ be integers as in the definition of the gluing as described above.
Define two homomorphisms $\N\to S_1$ and $\N\to S_2$
by sending $1$ to $i_1s_1+\cdots+i_ns_n$ and $j_1t_1+\cdots+j_mt_m$,
respectively. Define $S_1\to a S_1+ b S_2$ and $S_2\to a S_1+ b S_2$ by
multiplying $a$ and $b$ to $S_1$ an $S_2$, respectively.
The images of $1$ from $\N$ via the compositions of the above homomorphisms agree in $aS_1+bS_2$ by the gluing hypothesis on $S_1$ and $S_2$. These homomorphisms of affine
semigroups induce a commutative diagram
\[
\xymatrix{
\kappa[\bu^{S_1}] \ar[r]& \kappa[\bu^{aS_1+bS_2}] \\
\kappa[\bu^{\N}]   \ar[u]\ar[r]   &  \kappa[\bu^{S_2}] \ar[u]
}
\]
of affine semigroup rings.

Let $\mathcal R$ be an arbitrary $\kappa$-algebra.
Assume that the outer layer of the following diagram with $\mathcal R$ at the northeastern corner commutes.
\[
\xymatrix{
&& \mathcal R\\
\kappa[\bu^{S_1}] \simeq \kappa[\bu^{aS_1}] \ar[r]\ar@/^1pc/[rru]& \kappa[\bu^{aS_1+bS_2} ] \ar@{-->}[ru]\\
\kappa[\bu^{\N}] \ar[u]\ar[r] & \kappa[\bu^{S_2}] \simeq \kappa[\bu^{bS_2}] \ar[u]\ar@/_1pc/[ruu]  }
\]
By realizing the affine semigroup rings as quotients of polynomial rings modulo their corresponding defining ideals in the definition of gluing,
\[ \begin{array}{l}
\kappa[\bu^{S_1}] \simeq \kappa[\bu^{a S_1}] \simeq \kappa[X_1, \dots, X_n]/\ker \varphi_1 \\
\kappa[\bu^{S_2}] \simeq \kappa[\bu^{b S_2}] \simeq \kappa[Y_1, \dots, Y_m]/\ker \varphi_2 \\
\kappa[\bu^{aS_1+bS_2}] \simeq \kappa[X_1, \dots, X_n, Y_1, \cdots, Y_m]/ \ker \varphi,
\end{array} \]
we see that the images of $X_i$ and $Y_j$ in $\mathcal R$ from $\kappa[\bu^{aS_1}]$ and $\kappa[\bu^{bS_2}]$
naturally provide assignments of those variables in $\kappa[X_1, \dots, X_n, Y_1, \dots, Y_m]$.
While $\ker \varphi = \ker \varphi_1 + \ker \varphi_2 + \langle X^{i_1}_1\cdots X^{i_n}_n-Y^{j_1}_1\cdots Y^{j_m}_m\rangle$,
we notice that the images of $X^{i_1}_1\cdots X^{i_n}_n$ and $Y^{j_1}_1\cdots Y^{j_m}_m$ agree
for they are the image of $\bu$ from $\kappa[\bu^{\N}]$.
Thus both $\kappa$-algebra homomorphisms from $\kappa[\bu^{S_1}]$ and $\kappa[\bu^{S_2}]$ to $\mathcal R$
factor through $\kappa[\bu^{aS_1 + bS_2} ]$.
Such induced map must be unique by the desired commutativity of the resulting triangular diagrams.

The above work observes that $\kappa[\bu^{aS_1 + bS_2} ]$ satisfies the universal property
in the category of $\kappa$-algebras and concludes that
the gluing $\kappa[\bu^{aS_1+bS_2}]$ is isomorphic to $\tensor{R_1}{R_2}{R}$,
where $R:=\kappa[\bu^{\N}]$. This proves the first statement of the proposition

Since the algebra $R_2/R$ is flat, by Lemma~\ref{lem:05664},
\[
\kappa[\bu^{aS_1+bS_2}] \simeq \tensor{R_1}{R_2}{R}\simeq\kappa[\bu^{\fibersum{S_1}{S_2}{\N}}].
\]
This shows also that $\kappa[\bu^{\fibersum{S_1}{S_2}{\N}}]$ is an affine semigroup ring, so
the exponent $\fibersum{S_1}{S_2}{\N}$ must be torsion free which implies that it is isomorphic to $\Fibersum{S_1}{S_2}{\N}$.
Hence,
$\Tensor{R_1}{R_2}{R}\simeq\tensor{R_1}{R_2}{R}$.
\end{proof}

The next theorem recovers an equivalent condition for the existence of the
gluing as in \cite[Theorem 1.4]{ros:psN}.

\begin{thm}\label{glueEqCond}
Positive affine semigroups $S_1$ and $S_2$ in $\N^d$ can be glued
if and only if
there are $a,b \in \N$ and a non-zero element in $aS_1\cap bS_2$ generating $\gp(aS_1)\cap\gp(bS_2)$
in $\Z^d$.
\end{thm}

In fact, Theorem~\ref{glueEqCond} follows a special case when the gluing of $\kappa[\bu^{S_1}]$ and
$\kappa[\bu^{S_2}]$ takes place without the need of any scalar modification and is plainly $\kappa[\bu^{S_1+S_2}]$.
In Lemma~\ref{lem:gluingCond}, we further observe that the condition (\ref{glueKernel}) on the kernel for
$\kappa[\bu^{S_1+S_2}]$ to be the gluing of $\kappa[\bu^{S_1}]$ and $\kappa[\bu^{S_2}]$
in $\kappa[\bu]$ is equivalent to a condition on the groups generated by $S_1$ and $S_2$.
We begin proving this observation using a key Lemma~\ref{lem:emb}.

\begin{lem}\label{lem:gluingCond}
Let $S_1=\langle s_1,\ldots,s_n\rangle$ and $S_2=\langle t_1,\ldots,t_m\rangle$ be affine semigroups in $\N^d$,
and consider the $\kappa$-algebra homomorphisms
\begin{eqnarray*}
\varphi_1&:&\kappa[X_1,\ldots,X_n]\to\kappa[\bu^{S_1}],\\
\varphi_2&:&\kappa[Y_1,\ldots,Y_m]\to\kappa[\bu^{S_2}],\\
\varphi&:&\kappa[X_1,\ldots,X_n,Y_1,\ldots,Y_m]\to\kappa[\bu^{S_1+S_2}]
\end{eqnarray*}
given by $\varphi(X_i)=\varphi_1(X_i)=\bu^{s_i}$ and
$\varphi(Y_j)=\varphi_2(Y_j)=\bu^{t_j}$.
For any nonzero element $w = i_1s_1 + \cdots + i_ns_n = j_1t_1 + \cdots + j_mt_m$ in $S_1 \cap S_2$,
the following conditions are equivalent.
\begin{enumerate}[\em (i)]
\item $\ker\varphi=\ker\varphi_1+\ker\varphi_2+
\langle X^{i_1}_1\cdots X^{i_n}_n-Y^{j_1}_1\cdots Y^{j_m}_m\rangle$.
\item $\gp(S_1) \cap \gp(S_2) = \gp(S)$, where $S = \N w$.
\end{enumerate}
\end{lem}

\begin{proof}
First note that since $\fibersum{S_1}{S_2}{S}$ is a fibered sum, there exists a
canonical map $\fibersum{S_1}{S_2}{S}\to S_1+S_2$ by the universal property
and it induces a $\kappa$-algebra homomorphism
$\eta : \kappa[\bu^{\fibersum{S_1}{S_2}{S}}] \to \kappa[\bu^{S_1+S_2}]$.
Consider the $\kappa$-algebra homomorphism
\[
\widetilde{\varphi} : \kappa[X_1,\ldots,X_n,Y_1,\ldots,Y_m] \to \kappa[\bu^{\fibersum{S_1}{S_2}{S}}]
\]
given by $\widetilde{\varphi}(X_i)=\bu^{\fibersum{s_i}{0}{}}$ and
$\widetilde{\varphi}(Y_j)=\bu^{\fibersum{0}{t_j}{}}$.
Then $\varphi_1$, $\varphi_2$, $\varphi$, and $\widetilde{\varphi}$ are
epimorphisms, and we have the commutative diagram
\[
\xymatrix{
\kappa[X_1, \ldots, X_n, Y_1, \ldots, Y_m]  \ar[rr]^{\qquad\qquad \widetilde{\varphi}} \ar[rrd]_{\varphi} & & \kappa[\bu^{\fibersum{S_1}{S_2}{S}}]
\ar[d]^{\eta} \\
& & \kappa[\bu^{S_1+S_2}] .
}
\]
Next we consider the following ideal of $\kappa[X_1, \ldots, X_n, Y_1, \ldots, Y_m]$
\[
J = \ker\varphi_1 + \ker\varphi_2 + \langle X^{i_1}_1\cdots X^{i_n}_n-Y^{j_1}_1\cdots Y^{j_m}_m \rangle.
\]
Note that the image of $X^{i_1}_1\cdots X^{i_n}_n-Y^{j_1}_1\cdots Y^{j_m}_m$ under $\widetilde{\varphi}$ is
\[
\begin{array}{rl}
& \left(\bu^{\fibersum{s_1}{0}{}}\right)^{i_1} \cdots \left(\bu^{\fibersum{s_n}{0}{}}\right)^{i_n} - \left(\bu^{\fibersum{0}{t_1}{}}\right)^{j_1} \cdots \left(\bu^{\fibersum{0}{t_m}{}}\right)^{j_m}\\
= & \bu^{\fibersum{(i_1s_1+\cdots+i_ns_n)}{0}{}} - \bu^{\fibersum{0}{(j_1t_1+\cdots+j_mt_m)}{}} \\
= & \bu^{\fibersum{w}{0}{}} - \bu^{\fibersum{0}{w}{}}
\end{array}
\]
which is equal to $0$ since $w \in S$ and so $\fibersum{w}{0}{} = \fibersum{0}{w}{}$.
It is now clear that $J$ is contained in $\ker \widetilde{\varphi}$ and so $\widetilde{\varphi}$ induces a $\kappa$-algebra homomorphism
\[
\psi : \kappa[X_1, \ldots, X_n, Y_1, \ldots, Y_m]/J \to \kappa[\bu^{\fibersum{S_1}{S_2}{S}}]
\]
with $\ker \psi = \ker \widetilde{\varphi}/J$.
On the other hand, note that by Corollary~\ref{cor:flat} and Lemma~\ref{lem:05664},
we have an isomorphism
\[
\iota: \kappa[\bu^{\fibersum{S_1}{S_2}{S}}] \xrightarrow{\,\,\,\sim\,\,\,}
\tensor{\kappa[\bu^{S_1}]}{\kappa[\bu^{S_2}]}{\kappa[\bu^S]}.
\]
And also, due to the hypothesis that $S=\N w$ for $w \in S_1 \cap S_2$,  and that
$J$ contains $X^{i_1}_1\cdots X^{i_n}_n-Y^{j_1}_1\cdots Y^{j_m}_m$,
the following diagram commutes
\[
\xymatrix{
\kappa[\bu^{S_1}]  & \hspace{-16mm} \simeq  & \hspace{-16mm}\dfrac{\kappa[X_1, \ldots, X_n]}{\ker\varphi_1}  \ar[rrr]  & & &
\dfrac{\kappa[X_1, \ldots, X_n,Y_1,\ldots,Y_m]}{J} \\
\kappa[\bu^{S}] \ar[u]\ar[rrrrr] && & & & \kappa[\bu^{S_2}] \simeq \dfrac{\kappa[Y_1, \ldots, Y_m]}{\ker\varphi_2} . \ar[u]
}
\]
We will show in the remaining proof that this commutative diagram leads to the desired equivalence.

First, the universal property of tensor products  ensures the existence of the homomorphism
\[
\tensor{\kappa[\bu^{S_1}]}{\kappa[\bu^{S_2}]}{\kappa[\bu^S]} \to
\frac{\kappa[X_1,\ldots,X_n,Y_1,\ldots,Y_m]}{J},
\]
which then induces a homomorphism from an isomorphic $\kappa$-algebra
\[
\phi : \kappa[\bu^{\fibersum{S_1}{S_2}{S}}] \to
\frac{\kappa[X_1,\ldots,X_n,Y_1,\ldots,Y_m]}{J}.
\]
We observe the images of $X_i $ and $Y_j $ in $\kappa[X_1,\ldots,X_n,Y_1,\ldots,Y_m]/J$
via $\phi \circ \psi$
\[
\begin{array}{rl}
 & X_i \mapsto \bu^{\fibersum{s_i}{0}{}} \stackrel{\iota}{\mapsto} \bu^{s_i} \otimes 1  \mapsto X_i  \\
& Y_j  \mapsto \bu^{\fibersum{0}{t_j}{}} \stackrel{\iota}{\mapsto} 1 \otimes \bu^{t_j}   \mapsto Y_j
\end{array}
\]
and the images of $\bu^{\fibersum{s_i}{0}{}}$ and $\bu^{\fibersum{0}{t_j}{}}$ in
$\kappa[\bu^{\fibersum{S_1}{S_2}{S}}] $ via $\psi \circ \phi$
\[
\begin{array}{rl}
& \bu^{\fibersum{s_i}{0}{}} \stackrel{\iota^{\tiny{-1}}}{\longmapsto} \bu^{s_i} \otimes 1 \mapsto X_i  \mapsto \bu^{\fibersum{s_i}{0}{}} \\
& \bu^{\fibersum{0}{t_j}{}} \stackrel{\iota^{\tiny{-1}}}{\longmapsto} 1 \otimes \bu^{t_j} \mapsto Y_j  \mapsto \bu^{\fibersum{0}{t_j}{}}
\end{array}.
\]
This shows that $\phi$ and $\psi$ are the inverse of each other, and so $\psi$ is an isomorphism.  This implies $\ker \widetilde{\varphi} = J$ since $\ker \psi = \ker \widetilde{\varphi}/J$.

Moreover, the previous commutative diagram $\varphi = \eta \circ \widetilde{\varphi}$ shows that $\ker \varphi \supset \ker\widetilde{\varphi} = J$. The equality holds, which means
\[
\ker \varphi =\ker\varphi_1+\ker\varphi_2+
\langle X^{i_1}_1\cdots X^{i_n}_n-Y^{j_1}_1\cdots Y^{j_m}_m \rangle,
\]
if and only if $\eta$ is injective, which, by the construction of $\eta$ in the beginning of the proof,
is equivalent to that the canonical map $\fibersum{S_1}{S_2}{S}\to S_1+S_2$ is one-to-one.
Finally, by Lemma~\ref{lem:emb}, the last condition is equivalent to
$\gp(S_1) \cap \gp(S_2) = \gp(S)$.  This proves the lemma.
\end{proof}

Now we prove Theorem~\ref{glueEqCond}.

\begin{proof}
Assume that $R_1 \simeq \kappa[\bu^{S_1}]$ and $R_2\simeq \kappa[\bu^{S_2}]$ in some $\kappa[\bu]$
can be glued along a polynomial ring in one variable.
Then by definition, there exist positive integers $a$ and $b$ such that their gluing is $\kappa[\bu^{aS_1+ bS_2}]$ up to isomorphisms.
The forward implication follows immediately by applying Lemma~\ref{lem:gluingCond} to $aS_1$ and $bS_2$.

Conversely, if there exist positive integers $a$ and $b$ that satisfy the stated condition for the groups,
then we have condition (ii) in Lemma~\ref{lem:gluingCond} for $aS_1$ and $bS_2$.
Hence, its equivalent condition (i) holds which gives the gluing of $S_1$ and $S_2$.
\end{proof}

The result in Lemma~\ref{lem:gluingCond} implies that the condition of $\gp(aS_1) \cap \gp(bS_2)$ being generated by a single
element $w$ in $aS_1 \cap bS_2$ is a necessary and sufficient condition for gluing to take place.
That is, the existence of such $w$ determines whether or not a gluing is possible.
Furthermore, although gluing happens in the level of semigroups,
it is closely related to the relative positions of the groups generated by them.
In particular if an element in $S_1 \cap S_2$ generates $\gp(S_1) \cap \gp(S_2)$, then it generates $S_1 \cap S_2$,
but the converse does not hold.
For instance, let $S_1 = \langle (1,0), (1,1) \rangle$ and $S_2 = \langle (0,1), (1,1) \rangle$ in $\N^2$.
Then for any positive integers $a$ and $b$, $a S_1 \cap b S_2$ is always generated by one element, but $\gp(aS_1) \cap \gp(bS_2)$ has rank $2$. This reflects the folding intuition described after Example~\ref{ex:folding}. While condition (ii) in Lemma~\ref{lem:gluingCond}
cannot be established by scalar multiples, we see that it may be done by matrix transformations in the discussion following Example~\ref{ex:normalcurve2}.

{\bf Summary.}
We conclude this paper by reiterating our ``relative" point of view.
The idea of gluing is to construct a new affine semigroup ring from two given ones by identifying selected subrings.
We have seen two approaches in this section: from set-theoretic viewpoint or from categorical viewpoint.

The set-theoretic approach is the classical method initiated in Rosales~\cite{ros:psN}, 
in which the part to be identified is a one-dimensional polynomial ring.
To assure that no extra identification occurs outside the polynomial ring,
this classical approach takes into account the defining relations of the given affine semigroup rings and those of the new one.
Precisely, the new affine semigroup ring, if constructible, is restricted to acquire only one extra relation among the generators
in addition to the existing relations in the given affine semigroup rings.

From another standpoint, our approach of taking fibered sum is functorial, where an affine semigroup ring is regarded as an object equipped with morphisms in the category of affine semigroup rings.
An affine semigroup ring containing the two given ones with identified parts fits in a commutative diagram. The desired construction is a diagram, smallest in the sense that all such diagrams emerge from it. 
 (This is as illustrated in the proof of Theorem~\ref{thm:glueVfibersum}.)
So what we seek is a universal object among these diagrams, or in other words, a fibered sum in the category of affine semigroup rings. 

Consideration of fibered sums yields the generality that the parts to be identified can be an arbitrary affine semigroup ring. This potentially widens the range of future studies in the field. 
For example, in the numerical cases, singularities are stable under gluing and this fact can be explained by base changes \cite{hu-kim:nsa}. 
The fibered sum introduced in this paper, thus, calls for the investigation of singular properties of affine semigroup algebras under base changes, among many prospective directions.

\medskip

\noindent {Acknowledgement.} The first author would like to express her gratitude to the Institute of Mathematics of Academia Sinica for their generous support and hospitality during her multiple visits.

\end{document}